\documentclass%[dvipdfmx]
{article}
\usepackage{amsmath,amssymb,pst-all,graphicx}
\usepackage{amscd, amsthm}
\usepackage[cmtip,arrow]{xy}
\usepackage{pb-diagram, pb-xy}
\usepackage{bm}
\usepackage{enumitem}
\usepackage{tikz-cd}
\usepackage{color}
\usepackage{titlefoot}
\usepackage{algorithmic}
\usepackage{algorithm}
\usepackage{authblk}

\newcommand{\Aa}{\mathbb A}
\newcommand{\CC}{\mathbb C}

\newcommand{\PP}{\mathbb P}

\newcommand{\ZZ}{\mathbb Z}

\newcommand{\mcE}{{\mathcal E}}

\newcommand{\mcH}{\mathcal{H}}
\newcommand{\mcJ}{\mathcal{J}}

\newcommand{\mcL}{\mathcal{L}}
\newcommand{\mcM}{\mathcal{M}}
\newcommand{\mcN}{\mathcal{N}}
\newcommand{\mcO}{\mathop {\mathcal O}\nolimits}

\newcommand{\gtA}{\mathfrak{A}}
\newcommand{\gtB}{\mathfrak{B}}
\newcommand{\gtI}{\mathfrak{I}}
\newcommand{\gtU}{\mathfrak{U}}

\newcommand{\HH}{H}

\newcommand{\Ima}{\mathop {\rm Im}\nolimits}

\newcommand{\GL}{\mathop {\rm GL}\nolimits}

\newcommand{\PGL}{\mathop {\rm PGL}\nolimits}

\newcommand{\Pic}{\mathop {\rm Pic}\nolimits}
\newcommand{\Proj}{\mathop {\rm \bf Proj}\nolimits}

\newcommand{\Sing}{\mathop {\rm Sing}\nolimits}

\newcommand{\codim}{\mathop {\rm codim}\nolimits}

\newcommand{\Aut}{\mathop {\rm {Aut}}\nolimits}

\newcommand{\id}{\mathop {{\rm id}}\nolimits}
\newcommand{\pic}{\mathop{\rm Pic}\nolimits}

\newcommand{\wcl}{\mathrm{Cl}}
\newcommand{\Gr}{\mathrm{Gr}}

\newcommand{\natu}{ {\star} }
\newcommand{\ad}{\mathrm{AD}}

\newcommand{\sbpic}{\mathrm{sPic}}

\newcommand{\Eu}{\mathop \mathrm{Eu}\nolimits}
\newcommand{\gd}{\mathop \mathrm{lcd}\nolimits}
\newcommand{\den}{\mathop \mathrm{D}\nolimits}
\newcommand{\neu}{\mathop \mathrm{N}\nolimits}
\newcommand{\quot}{\mathop \mathrm{quot}\nolimits}

\newcommand{\rev}{}%parts of revision

\newtheorem{thm}{Theorem}[section] % 1st argument is your name for it
\newtheorem{Thm}[thm]{Theorem}
\newtheorem{lem}[thm]{Lemma}     % 2nd argument is what is printed
\newtheorem{cor}[thm]{Corollary}
\newtheorem{prop}[thm]{Proposition}
\newtheorem{conj}[thm]{Conjecture}

\theoremstyle{definition}
\newtheorem{defin}[thm]{Definition}
\newtheorem{rem}[thm]{Remark}

\newtheorem{ex}[thm]{Example}
\newtheorem{prob}[thm]{Problem}
\begin{document}

\title{Double covers and vector bundles of rank two}
\author{Taketo Shirane\thanks{
Partially supported by Grant-in-Aid for Scientific Research C (21K03182).
%Department of Mathematical Sciences, Faculty of Science and Technology, Tokushima University, 2-1 Minamijyousanjima-cho, Tokushima 770-8506, JAPAN. E-mail: {\tt shirane@tokushima-u.ac.jp}, ORCID iD:  0000-0002-4531-472X.
}
}
\affil{\small Department of Mathematical Sciences, Faculty of Science and Technology, Tokushima University, 2-1 Minamijyousanjima-cho, Tokushima 770-8506, JAPAN. \\ \textit{E-mail address:} {\tt shirane@tokushima-u.ac.jp}, ORCID iD:  0000-0002-4531-472X.}

\amssubj{primary: 14E20, secondary: 14C22, 14J60.}
\date{
%\today
}
\maketitle

\begin{abstract}
In 2017, Catanese--Perroni gave a natural correspondence between the Picard group of a double cover and a set of pairs of a $2$-bundle and a certain morphism of $2$-bundles on the base space. 
In this paper, we describe the group structure of the latter set induced from the Picard group in terms of transition functions of $2$-bundles. 
This study is derived from the study of embedded topology of plane curves. 
It also proposes approaches to the study of Picard groups of double covers, and to the construction of $2$-bundles. 
\end{abstract}

\section{Introduction}
In the study of the embedded topology of curves on the complex projective plane $\PP^2$, it is effective to consider the irreducibility of $\phi^\ast C$ for an irreducible curve $C\subset\PP^2$ and a Galois cover $\phi:X\to\PP^2$ (cf. \cite{bannai2016}, \cite{shirane2017}, \cite{shirane2019}). 
For example, let $B,C\subset\PP^2$ be two plane curves such that $\deg B$ is even and $C$ is irreducible with $\deg B\ne\deg C$, and let $\phi:X\to \PP^2$ be the double cover branched at $B$; 
then the embedded topology of $B+C$ changes depending on whether $\phi^\ast C$ is irreducible or not. 
In the case where $\phi$ is a cyclic cover and $C$ is smooth, a criterion for irreducibility of $\phi^\ast C$ is known in \cite{benshi2017}. 
This criterion is intensively used to distinguish embedded topology of plane curves (cf. \cite{abst}, \cite{abst_flex}, \cite{bgst2017}, \cite{shirane2019}). 
In the case where $\phi$ is the double cover branched at a smooth conic and $C$ is a nodal curve, a criterion for irreducibility of $\phi^\ast C$ is known in \cite{banshi2019}. 
However, for general $C\subset\PP^2$, it is still a problem to determine the irreducibility of $\phi^\ast C$ even if $\phi$ is a double cover. 
We consider a new approach to the problem, which is constructing various curves on $X$ (which correspond to irreducible components of $\phi^\ast C$) and studying property of their images (which correspond to $C$). 
The main aim of this paper is a preparation for this new approach by studying a correspondence between line bundles on $X$ and vector bundles of rank $2$ (say \textit{$2$-bundles} for short) on $Y$ in the case where $\phi$ is a double cover. 
As bi-products, we obtain approaches to studying the Picard group of double covers and to constructing $2$-bundles. 

In this paper, we consider a \textit{non-singular double cover} $\phi:X\to Y$ which is a finite surjective morphism of degree $2$ between non-singular varieties $X$ and $Y$ of any dimension over $\CC$. 
Catanese--Perroni \cite{catper2017} gave a natural correspondence between $\Pic(X)$ and a set of pairs of a $2$-bundle on $Y$ and a certain morphism of $2$-bundles on $Y$ (we call such pairs \textit{admissible pairs} for $\phi$). 
Hence the latter set has a group structure induced by that of $\Pic(X)$. 
Our aim is to give explicit formulas for the group structure in terms of transition functions of $2$-bundles. 
In particular, we give a criterion if $\phi_\ast\mcL$ splits (i.e., a direct summand of two line bundles on $Y$). 

A non-singular double cover $\phi:X\to Y$ is determined by the branch locus $B_\phi\subset Y$ and a divisor $L$ on $Y$ such that $2L\sim B_\phi$. 
However the relation between the pair $(B_\phi,L)$ and $\pic(X)$ is unknown. 
We will define a subgroup $\sbpic_\phi(X)$ of $\pic(X)$ generated by line bundles $\mcL$ such that $\phi_\ast\mcL$ splits. 
We can guess from the condition for $\phi_\ast\mcL$ to split that the structure of $\sbpic_\phi(X)$ strongly relate with the embedding $B_\phi\subset Y$ (this relation is still a problem). 
The author conjectures that the equality $\sbpic_\phi(X)=\pic(X)$ holds. 
In fact, the equality holds in several simple cases. 
Recently, it is proved in \cite{shirane2021} that the equality holds if $Y$ is isomorphic to the projective space $\PP^n$. 

From the view point of $2$-bundles, the group structure of $\pic(X)$ gives us a method for constructing $2$-bundles on $Y$, i.e., we can compute the transition functions of the $2$-bundle $\phi_\ast(\mcL_1^{\pm1}\otimes\dots\otimes\mcL_m^{\pm1})$ if line bundles $\mcL_k$ are known. 
Schwarzenberger \cite{schwarzenberger1961} proved that, for any $2$-bundle $\mcE$ on a smooth surface $Y$, there exist a non-singular double cover $\phi:X\to Y$ and a line bundle $\mcL$ on $X$ such that $\phi_\ast\mcL\cong\mcE$ (\cite[Theorem~3]{schwarzenberger1961}). 
In Section~\ref{sec:normal}, we will show a generalization of \cite[Theorem~3]{schwarzenberger1961} (Theorem~\ref{thm:push}). 
From this generalization, we can also expect that any $2$-bundle on a smooth variety can be constructed by our method. 

The main theorems of this paper are Theorems~\ref{thm:group_law0}, \ref{thm:split} and \ref{thm:push} below. 
In Theorems~\ref{thm:group_law0} and \ref{thm:split}, let $\phi:X\to Y$ be a non-singular double cover with $\phi_\ast\mcO_X\cong\mcO_Y\oplus\mcO_Y(-L)$ for a divisor $L$ on $Y$, and let $\iota:X\to X$ be the covering transformation of $\phi$. 
Theorem~\ref{thm:group_law0} describes the group structure of $\pic(X)$ in terms of $2$-bundles on $Y$.  

\begin{Thm}\label{thm:group_law0}
	Let $\mcL_k$ $(k=1,\dots,m)$ be $m$ line bundles on $X$. 
	Then there exist an affine open covering $\{U_i\}_{i\in I}$ of $Y$ and $K_{ij}^{(k)+}, K_{ij}^{(k)-}\in\GL(2,\mcO_{U_i\cap U_j})$ for any $i,j\in I$ and $k=1,\dots,m$ such that 
	the matrices 
	\[ \left(K_{ij}^{(1)}(n_1)\right)^{|n_1|}\dots\left(K_{ij}^{(m)}(n_m)\right)^{|n_m|}
	\begin{pmatrix}
	1 & 0 \\ 0 & \xi_{ij}
\end{pmatrix} 
\in \GL(2,\mcO_{U_i\cap U_j}) \quad (i,j\in I)
 \]
form transition functions of $\phi_\ast(\mcL_1^{n_1}\otimes\dots\otimes\mcL_m^{n_m})$ for any $n_1,\dots,n_m\in\ZZ$,
where $K_{ij}^{(k)}(n_k):=K_{ij}^{(k)+}$ if $n_k\geq0$, $K_{ij}^{(k)}(n_k):=K_{ij}^{(k)-}$ if $n_k<0$, and $\xi_{ij}\in\Gamma(U_i\cap U_j,\mcO_Y^\times)$ correspond to transition functions of $\mcO_Y(-L)$. 
Moreover, the explicit formulas for $K_{ij}^{(k)+}$ and $K_{ij}^{(k)-}$ are given in Theorem~\ref{thm:group_law}. 
\end{Thm}

\begin{rem}
The push forward $\phi_\ast(\mcL_1^{n_1}\otimes\dots\otimes\mcL_m^{n_m})$ may be indecomposable even if $\phi_\ast\mcL_i$ splits for each $i=1,\dots,m$ (see Example~\ref{ex:conic} and \ref{ex:quartic}). 
\end{rem}

Theorem~\ref{thm:split} is a criterion for splitting of the push-forward for a line bundle on $X$ (see Remark~\ref{rem:generator} and Lemma~\ref{lem:generator}). 

\begin{Thm}\label{thm:split}
Let $D^+$ be an effective divisor on $X$, and
%Let $D'$ be a divisor on $Y$, and 
let $D$ be the effective divisor on $Y$ defined by $f=0$ for $f\in\HH^0\big(Y,\mcO_Y(D)\big)$ such that $\phi^\ast D=D^++\iota^\ast D^+$. 
Assume that $\HH^0(Y,\mcO_Y)=\CC$. 

%\begin{enumerate}[label=\rm (\alph{enumi})]
%	\item $\HH^0(Y,\mcO_Y)=\CC$; 
%	\item either $\mcO_Y(D')\cong\mcO_Y$ or $\HH^0\big(Y,\mcO_Y(D')\big)=0$; 
%	\item $D$ is defined by $f=0$ for $f\in\HH^0\big(Y,\mcO_Y(D)\big)$. 
%\end{enumerate}
If $\phi_\ast\mcO_X(D^+)\cong\mcO_Y(D')\oplus\mcO_Y$ for a divisor $D'$ on $Y$ satisfying either $\mcO_Y(D')\cong\mcO_Y$ or $\HH^0\big(Y,\mcO_Y(D')\big)=0$, 
then $D$ and $D'$ satisfy the following two conditions{\rev :} 
\begin{enumerate}[label=\rm (\roman{enumi})]
	\item $D'$ is linearly equivalent to $D-L$, i.e., $\mcO_Y(D')\cong\mcO_Y(D-L)$; and 
	\item there are global sections $a_0$ and $a_1$ of $\mcO_Y(L)$ and $\mcO_Y(2L-D)$, respectively, such that $F=a_0^2+fa_1$, {\rev where $F$ is a global section of $\mcO_Y(B_\phi)$ defining the branch locus $B_\phi$ of the double cvore $\phi$.} 
\end{enumerate}
Moreover, in the case where $D^+$ is irreducible, if $D$  {\rev satisfies} {\rm (ii)}, then $\phi_\ast\mcO_X(D^+)\cong\mcO_Y(D-L)\oplus\mcO_Y$. 
%the converse holds true. 
\end{Thm}

Theorem~\ref{thm:split} shows a correspondence between line bundles on a smooth double cover and equations of the form $F=a_0^2+a_1a_2$. 
In the case of $Y=\PP^1$ (hence $X$ is a hyperelliptic curve), Jacobi \cite{jacobi1846} have studied this correspondence via the Jacobian variety (cf. \cite{terasoma2019}). 

A finite surjective morphism of degree $2$ from a normal variety to a smooth variety is called a \textit{normal} double cover in this paper. 
The following theorem is a generalization of \cite[Theorem~3]{schwarzenberger1961}. 

\begin{Thm}\label{thm:push}
Let $\mcE$ be a $2$-bundle on a smooth projective variety $Y$ of dimension $n$ over $\CC$. 
There exist a normal double cover $\phi:X\to Y$ and a divisorial sheaf $\mcL$ on $X$ such that $\mcE\cong\phi_\ast\mcL$.
\end{Thm}

This paper is organized as follows. 
In Section~\ref{sec:2-bundles}, we rewrite the correspondence in \cite{catper2017} to make clear the relation between sections of $\mcL$ and $\phi_\ast\mcL$ for a line bundle $\mcL$ on $X$ (Proposition~\ref{prop:corr} and Corollary~\ref{cor:corr}). 
In Section~\ref{sec:group_structure}, we prove Theorem~\ref{thm:group_law0} (see Theorem~\ref{thm:group_law}). 
In Section~\ref{sec:subgroup}, we prove Theorem~\ref{thm:split}. 
We also define a subgroup $\sbpic_\phi(X)\subset\Pic(X)$ generated by line bundles on $X$ whose push-forwards by $\phi$ split. 
Observing several examples, a conjecture arises (Conjecture~\ref{conj:spic}). 
In Section~\ref{sec:normal}, we prove Theorem~\ref{thm:push}, and give an idea to generate $2$-bundles. 
In Section~\ref{sec:P1}, we give algorithms for computing the direct summands of a $2$-bundle on $\PP^1$ from its transition functions. 
%For this aim, we prove the well-known result due to A.~Grothendieck by using transition functions, which claims that any $2$-bundle on $\PP^1$ splits. 
In Section~\ref{sec:example}, we prove \cite[Proposition~8]{schwarzenberger1961} about jumping lines of $\phi_\ast\mcL$ by the method in Section~\ref{sec:P1} in the case where $\phi:X\to\PP^2$ is a non-singular double cover branched at conic on $\PP^2$. 
Moreover, we compute certain global sections of a line bundle on $X$ (Example~\ref{ex:global_section}).

\section{Line bundles over a double cover}\label{sec:2-bundles}

Let $\phi:X\to Y$ be a double cover, where $Y$ is a smooth variety. 
In \cite{catper2017}, Catanese and Perroni investigated a correspondence between line bundles on $X$ and $2$-bundles on $Y$. 
In this section, we describe it in terms of transition functions of $2$-bundles in the case where $X$ and $Y$ are smooth. 
We call a double cover $\phi:X\to Y$ a \textit{non-singular double cover} if $X$ and $Y$ are smooth over $\CC$. 

Let $\phi:X\to Y$ be a non-singular double cover. 
Let $B_\phi\subset Y$ and $R_\phi\subset X$ denote the branch locus and the ramification locus of $\phi$, respectively, and let $F\in\HH^0(Y,\mcO_Y(B_\phi))$ be a section defining $B_\phi$. 
There exists a divisor $L$ on $Y$ such that $2L\sim B_\phi$ and $\phi_\ast\mcO_X\cong\mcO_Y\oplus\mcO_Y(-L)$ as $\mcO_Y$-modules. 
We also have 
\begin{align*}
	\phi_\ast\mcO_X\cong \left(\bigoplus_{n=0}^\infty\, t^n\mcO_Y(-nL)\right)\Bigg/(t^2-F) 
\end{align*}
as $\mcO_Y$-algebras. 
Here $t$ corresponds to a section of $\mcO_X(R_\phi)$ defining the ramification locus $R_\phi$. 
{\rev 
For a line bundle $\mcL$ on $X$, $\phi_\ast\mcL$ is a $\phi_\ast\mcO_X$-module. 
Hence the multiplication by $t$ gives a morphism $M_\mcL:\phi_\ast\mcL(-L)\to\phi_\ast\mcL$ such that $M_\mcL^2=F{\cdot}\id_{\phi_\ast\mcL}$, i.e., the composition of $M_\mcL(-L):\phi_\ast\mcL(-2L)\to\phi_\ast\mcL(-L)$ and $M_\mcL:\phi_\ast\mcL(-L)\to\phi_\ast\mcL$ is the multiplication by $F$. 
}

\begin{defin}
	Let $(\mcM,M)$ be a pair of a $2$-bundle $\mcM$ on $Y$ and a morphism $M:\mcM(-L)\to\mcM$. 
	\begin{enumerate}
	\item We call $(\mcM,M)$ an \textit{admissible pair for $\phi$} if $M^2=F{\cdot}\id_\mcM:\mcM(-2L)\to\mcM$. 
	%i.e., the composition of $M(-L):\mcM(-2L)\to\mcM(-L)$ and $M:\mcM(-L)\to\mcM$ is the multiplication by $F$. 
	\item Two admissible pairs $(\mcM,M)$ and $(\mcN,N)$ are \textit{equivalent} if there exists an isomorphism $\Psi:\mcM\to\mcN$ such that $\Psi\circ M=N\circ\Psi(-L)$, and write $(\mcM,M)\sim(\mcN,N)$. 
	\end{enumerate}
\end{defin}

Let $\ad_\phi(Y)$ be the set of all equivalence classes of admissible pairs for a non-singular double cover $\phi:X\to Y$: 
\begin{align*}
	\ad_\phi(Y):=\big\{(\mcM,M):\mbox{ an admissible pair for $\phi$}\big\}\big/\sim 
\end{align*}

To describe a correspondence between admissible pairs for a non-singular double cover $\phi:X\to Y$ and line bundles on $X$, we introduce some notation. 
Let $(\mcM,M)$ be an admissible pair for $\phi$, and let $\gtU:=\{U_i\}_{i\in I}$ be an affine open covering of $Y$ such that 
\begin{align*}%\label{eq:open covering} 
	\begin{aligned}
	\phi_\ast\mcO_{X}|_{U_i} &\cong \mcO_{U_i}\oplus\mcO_{U_i} t_i & \mbox{as $\mcO_{U_i}$-algebras,} 
	\\
	\varphi_i:\mcM_i &:=\mcM|_{U_i} \overset{\sim}{\to}\mcO_{U_i}^{\oplus2} & \mbox{as $\mcO_{U_i}$-modules}
	\end{aligned}
\end{align*}
for any $i\in I$, where %$t\in\HH^0(X,\mcO_X(R_\phi))$ defines $R_\phi$ (hence $t^2=F$), and 
$t_i:=t|_{U_i}$. 
Note that 
%\begin{align}\label{eq:ti}
$t_j=t_i\xi_{ij}$ 
%\end{align}
for $i,j\in I$, 
where $\xi_{ij}\in\mcO_{U_i\cap U_{j}}^\times$ correspond to transition functions of $\mcO_Y(-L)$:
\[
\begin{tikzpicture}
	\node (Lj) at (0,1.2) {$\mcO_Y(-L)|_{U_j}$};
	\node (Li) at (3,1.2) {$\mcO_Y(-L)|_{U_i}$};
	\node (Oj) at (0,0) {$\mcO_{U_j}$};
	\node (Oi) at (3,0) {$\mcO_{U_i}$};
	
	\draw[double distance=2.5pt] (Lj)--(Li);
	\draw[->] (Oj)-- node[above, scale=0.7] {$\times \xi_{ij}$} (Oi);
	\draw[->] (Lj)-- node[right] {\rotatebox{90}{$\sim$}} (Oj);
	\draw[->] (Li)-- node[right] {\rotatebox{90}{$\sim$}} (Oi);
\end{tikzpicture}
\]
Then we have transition functions $G_{ij}\in\GL(2,\mcO_{U_i\cap U_{j}})$ of $\mcM$ for $i,j\in I$: 
\begin{align}\label{eq:Gij}
	G_{ij}=
	\begin{pmatrix}
		g_{ij,11} & g_{ij,12} \\ g_{ij,21} & g_{ij,22}
	\end{pmatrix}:=\varphi_i\circ\varphi_j^{-1}
	:\mcO_{U_j}^{\oplus 2}|_{U_i\cap U_{j}}\to\mcO_{U_i}^{\oplus 2}|_{U_i\cap U_{j}}
\end{align}
satisfying $G_{ik}=G_{ij}G_{jk}$ and $G_{ii}=E$ for each $i,j,k\in I$, where $E$ is the identity matrix. 
The restriction of $M:\mcM(-L)\to\mcM$ to $U_i$ corresponds to a matrix $M_i$:
\begin{align}\label{eq:Mi}
	M_i=
	\begin{pmatrix}
		a_{i0} & a_{i2} \\ a_{i1} & -a_{i0}
	\end{pmatrix}:=\varphi_i\circ\big(\varphi_i(-L)\big)^{-1}
	:\mcO_{U_i}^{\oplus 2}\to\mcO_{U_i}^{\oplus2}
\end{align}
satisfying $a_{i0}^2+a_{i1}a_{i2}=F_i:=F|_{U_i}$ and $M_{j}=\xi_{ij}G_{ij}^{-1}M_{i}G_{ij}$ as elements of $\GL(2,\CC(X))$ for each $i,j\in I$:
\[
\begin{tikzpicture}
	\node (mjl) at (-3.5,1.5) {$\mcM_j(-L)$};
	\node (mj) at (3.5,1.5) {$\mcM_j$};
	\node (mil) at (-3.5,-1.5) {$\mcM_i(-L)$};
	\node (mi) at (3.5,-1.5) {$\mcM_i$};
	\node (oj1) at (-1,0.7) {$\mcO_{U_j}^{\oplus 2}$};
	\node (oj2) at (1,0.7) {$\mcO_{U_j}^{\oplus 2}$};
	\node (oi1) at (-1,-0.7) {$\mcO_{U_i}^{\oplus 2}$};
	\node (oi2) at (1,-0.7) {$\mcO_{U_i}^{\oplus 2}$};
	
	\draw[->] (mjl) -- node[midway,sloped,below,scale=0.7] {$\varphi_j(-L)$} (oj1);
	\draw[->] (mil) -- node[midway,sloped,above,scale=0.7] {$\varphi_i(-L)$} (oi1);
	\draw[->] (mj) -- node[midway,sloped,below,scale=0.7] {$\varphi_j$} (oj2);
	\draw[->] (mi) -- node[midway,sloped,above,scale=0.7] {$\varphi_i$} (oi2);
	\draw[->] (oj1) -- node[below, scale=0.7] {$M_{j}$} (oj2);
	\draw[->] (oi1) -- node[above, scale=0.7] {$M_{i}$} (oi2);
	\draw[->] (oj1) -- node[left, scale=0.7] {$\xi_{ij}G_{ij}$} (oi1);
	\draw[->] (oj2) -- node[right, scale=0.7] {$G_{ij}$} (oi2);
	\draw[double distance=2.5pt] (mjl) -- (mil);
	\draw[double distance=2.5pt] (mj) -- (mi);
	\draw[->] (mjl) -- node[below, scale=0.7] {$M$} (mj);
	\draw[->] (mil) -- node[above, scale=0.7] {$M$} (mi);
\end{tikzpicture}
\]

\begin{defin}
With the above notation, we call $(\{G_{ij}\},\{M_i\})_\gtU$ a \textit{representation} of the admissible pair $(\mcM,M)$. 
A representation $(\{G_{ij}\},\{M_i\})_{\gtU}$ of $(\mcM,M)$ is said to be \textit{good} if $a_{i1}$ is a unit on $U_i$ for each $i\in I$. 
\end{defin}

\begin{lem}\label{lem:good-rep}
Any admissible pair $(\mcM,M)$ for $\phi$ has a good representation. 
\end{lem}
\begin{proof}
Let $(\{G_{ij}\},\{M_i\})_\gtU$ be a representation of $(\mcM,M)$, where $G_{ij}$ and $M_{i}$ are as (\ref{eq:Gij}) and (\ref{eq:Mi}), respectively. 
Note that we have 
\begin{align*}
	a_{i1}&\ne0, & a_{i2}&\ne0, & \{a_{i0}&=a_{i1}=a_{i2}=0\}=\emptyset 
\end{align*}
on $U_i$ since $X$ is smooth. 
Hence, if we put $a_{ik}:=a_{i1}-2p_{ik}a_{i0}-p_{ik}^2a_{i2}$ for general sections $p_{ik}$ ($3\leq k\leq n_i$) of $\mcO_{U_i}$ for some $n_i\geq 3$, 
we obtain 
\begin{align*}
	\{a_{i1}=a_{i2}=a_{i3}=\dots=a_{in_i}=0\}=\emptyset 
\end{align*}
on $U_i$. 
Let $I^\natu:=\{(i,k)\mid i\in I,\ k=1,\dots,n_i\}$. 
We put $A_\alpha$ for $\alpha=(i,k)\in I^\natu$ as 
\begin{align}\label{eq:Palpha}
	A_{(i,1)}&:=E, 
	&
	A_{(i,2)}&:=
	\begin{pmatrix}
		0 & 1 \\ 1 & 0
	\end{pmatrix}, 
	&
	A_{(i,k)}&:=
	\begin{pmatrix}
		1 & 0 \\ p_{ik} & 1
	\end{pmatrix} \ \ (3\leq k\leq n_i), 
\end{align}
and define $g^\natu_{\alpha\beta,11},\dots,g^\natu_{\alpha\beta,22}$ and $a^\natu_{\alpha0},a^\natu_{\alpha1},a^\natu_{\alpha2}$ for $\alpha=(i,k), \, \beta=(j,k')\in I^\natu$ by 
\begin{align*}%\label{eq:GM_natural}
	\begin{aligned}
	G_{\alpha\beta}^\natu
	&:=
	A_{\alpha}^{-1}G_{ij} A_\beta
	=
	\begin{pmatrix}
		g_{\alpha\beta, 11}^\natu & g_{\alpha\beta, 12}^\natu \\ g_{\alpha\beta, 21}^\natu & g^\natu_{\alpha\beta, 22}
	\end{pmatrix},
	\\
	M_\alpha^\natu
	&:=
	A_{\alpha}^{-1}M_i A_\alpha
	=
	\begin{pmatrix}
		a^\natu_{\alpha 0} & a^\natu_{\alpha 2} \\ a^\natu_{\alpha 1} & -a^\natu_{\alpha 0}
	\end{pmatrix}. 
	\end{aligned}
\end{align*}
Note that, since $a^\natu_{\alpha1}=a^\natu_{(i,k)1}=a_{ik}$ for $\alpha=(i,k)\in I^\natu$, we have 
\[ \{a^\natu_{(i,1)1}=a^\natu_{(i,2)1}=a^\natu_{(i,3)1}=\dots=a_{(i,n_i)1}^\natu=0\}=\emptyset. \] 
 Put $U_{\alpha}^\natu:=U_i\cap\{a^\natu_{\alpha 1}\ne0\}$ for $\alpha=(i, k)\in I^\natu$. 
Then $\gtU^\natu:=\{U_{\alpha}^\natu\}_{\alpha\in I^\natu}$ is an affine open covering of $Y$, and $(\{G_{\alpha\beta}^\natu\},\{M_{\alpha}^\natu\})_{\gtU^\natu}$ is a good representation of $(\mcM,M)$. 
\end{proof}

\begin{prop}[{\cite[Lemma~3.1 and Proposition~3.3]{catper2017}}]\label{prop:corr}
	Let $\phi:X\to Y$ be a non-singular double cover and let $\iota:X\to X$ be the covering transformation of $\phi$. 
	{\rev 
	The map $\Upsilon:\pic(X)\to\ad_\phi(Y)$ defined by 
		\begin{align*}
		\begin{array}{cccc}\Upsilon:& \pic(X) & \to & \ad_\phi(Y) \\ & \rotatebox{90}{$\in$} & & \rotatebox{90}{$\in$} \\ & [\mcL] & \mapsto & [(\phi_\ast\mcL,M_\mcL)]  \end{array}
		\end{align*} 
	 is well-defined and bijective. 
	}
%	Then 
%		there exists a natural one-to-one correspondence 
%		\begin{align*}
%		\begin{array}{cccc}\Upsilon:& \ad_\phi(Y) &\to &\pic(X) \\ & \rotatebox{90}{$\in$} & & \rotatebox{90}{$\in$} \\ & [(\mcM,M)] & \mapsto & [\mcL_{(\mcM,M)}] \end{array}
%		\end{align*} 
%		such that $\mcM\cong\phi_\ast\mcL_{(\mcM,M)}$. 
\end{prop}
\begin{proof}
{\rev
Let $\CC(X)$ be the constant sheaf of rational functions on $X$, which can be regarded as an $\mcO_Y$-algebra. 
Any line bundle $\mcL$ on $X$ can be canonically embedded in $\CC(X)$, and $\phi_\ast\mcL$ is $\mcL\subset\CC(X)$ regarded as an $\mcO_Y$-module. 
 If $\mcL\cong\mcL'$ for line bundles $\mcL,\mcL'\subset\CC(X)$ on $X$, then there is a rational function $q\in\CC(X)^\times$ such that $\mcL'=q\mcL$ in $\CC(X)$. 
 The multiplication by $q$ gives an isomorphism $\Psi_q:\phi_\ast\mcL\to\phi_\ast\mcL'$ with $\Psi_q\circ M_\mcL=M_{\mcL'}\circ\Psi_q(-L)$. 
 Hence $(\phi_\ast\mcL,M_\mcL)\sim(\phi_\ast\mcL',M_{\mcL'})$, and $\Upsilon$ is well-defined. 

If $(\phi_\ast\mcL,M_\mcL)\sim(\phi_\ast\mcL',M_{\mcL'})$ for two line bundles $\mcL,\mcL'$ on $X$, 
then there is an isomorphism $\Psi:\phi_\ast\mcL\to\phi_\ast\mcL'$ with $\Psi\circ M_{\mcL}=M_{\mcL'}\circ\Psi(-L)$. 
This implies that $\Psi$ preserves the $\mcO_X$-module structures of $\mcL$ and $\mcL'$. 
Hence $\mcL\cong\mcL'$, and $\Upsilon$ is injective. 

We show that $\Upsilon$ is surjective. 
}
Let $(\mcM,M)$ be an admissible pair for $\phi$, 
and let $(\{G_{ij}\},\{M_i\})_{\gtU}$ be a good representation of $(\mcM,M)$, where $G_{ij}$ and $M_i$ are as (\ref{eq:Gij}) and (\ref{eq:Mi}). 
%Let $\CC(X)$ be the constant sheaf of rational functions on $X$, and be regarded as an $\mcO_Y$-algebra. 
Fix an element $i_0\in I$. 
Let $\bm{d}_{i_01}, \bm{d}_{i_02}\in\CC(X)$ be linearly independent rational functions over $\CC(Y)$ satisfying 
\begin{align*}
	t_{i_0}
	\begin{pmatrix}
		\bm{d}_{i_01} & \bm{d}_{i_02}
	\end{pmatrix}
	=
	\begin{pmatrix}
		\bm{d}_{i_01} & \bm{d}_{i_02}
	\end{pmatrix}
	M_{i_0}.
\end{align*}
Note that there exist such $\bm{d}_{i_01},\bm{d}_{i_02}$, for example $\bm{d}_{i_01}=t_{i_0}+a_{i_00}$ and $\bm{d}_{i_02}=a_{i_02}$. 
%Let $\eta_{i_0}:\mcM_{i_0}\to\CC(X)$ be a homomorphism defined by 
%\begin{align} 
%	\eta_{i_0}\circ\varphi_{i_0}^{-1}\begin{pmatrix} s \\ t \end{pmatrix}&:=
%	\begin{pmatrix} \bm{d}_{i_01} & \bm{d}_{i_02} \end{pmatrix}\begin{pmatrix} s \\ t \end{pmatrix}. 
%\end{align}
%Let $\bm{d}_{i1}, \bm{d}_{i2}\in\CC(X)$ be rational functions given by 
%\begin{align}
%	\begin{pmatrix}
%		\bm{d}_{i1} & \bm{d}_{i2}
%	\end{pmatrix}
%	=
%	\begin{pmatrix}
%	\bm{d}_{i_01} & \bm{d}_{i_02}
%	\end{pmatrix}
%	G_{i_0i}
%\end{align}
Let $\eta_i:\mcM_i\to\CC(X)$ be the homomorphism defined by 
\begin{align*}
	\eta_i\circ\varphi_i^{-1}\begin{pmatrix} s_1 \\ s_2 \end{pmatrix}&:=
	\begin{pmatrix} \bm{d}_{i_01} & \bm{d}_{i_02} \end{pmatrix}G_{i_0i}\begin{pmatrix} s_1 \\ s_2 \end{pmatrix}. 
\end{align*}
We have $\eta_i=\eta_j$ on $U_i\cap U_j$ since $\varphi_j^{-1}=\varphi_i^{-1}\circ G_{ij}$ and 
\begin{align*}
\begin{aligned}
	\eta_j\circ\varphi_j^{-1}\begin{pmatrix} s_1 \\ s_2 \end{pmatrix}
	&=
	\begin{pmatrix} \bm{d}_{i_01} & \bm{d}_{i_02} \end{pmatrix}G_{i_0j}\begin{pmatrix} s_1 \\ s_2 \end{pmatrix}
	\\
	&=
	\begin{pmatrix} \bm{d}_{i_01} & \bm{d}_{i_02} \end{pmatrix}G_{i_0i}G_{ij}\begin{pmatrix} s_1 \\ s_2 \end{pmatrix}
	\\
	&=
	\eta_i\circ\varphi_{i}^{-1}\circ G_{ij}\begin{pmatrix} s_1 \\ s_2 \end{pmatrix}.
\end{aligned}
\end{align*}
Thus we can define an inclusion $\eta:\mcM\hookrightarrow\CC(X)$ of $\mcO_Y$-modules by gluing $\eta_i$ ($i\in I$).
%\begin{align}
%	\eta_i:=\eta_{i_0} \circ \varphi_{i_0}^{-1}\circ G_{i_0i} \circ \varphi_{i} %\circ \res^{U_i}_{U_i\cap U_{i_0}}, 
%\end{align}
%on $U_{i_0}\cap U_i$. %where $\res^{U_i}_{U_i\cap U_{i_0}}:\mcM_i\to\mcM_{i}|_{U_i\cap U_{i_0}}$ is the restriction map. 
The map $\eta$ is injective since $\bm{d}_{i_01}, \bm{d}_{i_02}$ are linearly independent over $\CC(Y)$. 
Put 
\begin{align*}
	\bm{d}_{i1}&:=\eta_i\circ\varphi_i^{-1} \begin{pmatrix} 1 \\ 0 \end{pmatrix}
	&
	\bm{d}_{i2}&:=\eta_i\circ\varphi_i^{-1} \begin{pmatrix} 0 \\ 1 \end{pmatrix}
\end{align*}
for each $i\in I$.
Then we have 
\begin{align}\label{eq:transform} 
%	\begin{aligned}
	\begin{pmatrix} \bm{d}_{j1} & \bm{d}_{j2} \end{pmatrix} 
	%=
	%\begin{pmatrix} \eta_j\begin{pmatrix}1 \\0 \end{pmatrix} & \eta_j\begin{pmatrix} 0\\1 \end{pmatrix} \end{pmatrix} 
	&=
	\begin{pmatrix}
	\bm{d}_{i_01} & \bm{d}_{i_02}
	\end{pmatrix}
	G_{i_0i}G_{ij}
	=
	\begin{pmatrix}
		\bm{d}_{i1} & \bm{d}_{i2}
	\end{pmatrix}
	G_{ij}. 
\end{align}
Moreover, $M:\mcM(-L)\to\mcM$ is compatible with the multiplication by $t$ under $\eta$, i.e., 
\begin{align*}%\label{eq:compatible}
	t_i
	\begin{pmatrix}
		\bm{d}_{i1} & \bm{d}_{i2}
	\end{pmatrix}
	&=
%	t_{i_0}\xi_{i_0i}
%	\begin{pmatrix}
%	\bm{d}_{i_01} & \bm{d}_{i_02}
%	\end{pmatrix}
%	G_{i_0i}
%	=
	\begin{pmatrix}
		\bm{d}_{i1} & \bm{d}_{i2}
	\end{pmatrix}
	M_i
	\qquad \mbox{ for each $i\in I$, }
%	\end{aligned}
\end{align*}
since $t_{i}=t_{i_0}\xi_{i_0i}$ and $M_i=\xi_{i_0i}G_{i_0i}^{-1}M_{i_0}G_{i_0i}$. %, where $\xi_{ij}\in\mcO_{U_i\cap U_{j}}^\times$ corresponds to the transformation of $\mcO_Y(L)$. 
Since $(\{G_{ij}\},\{M_i\})_\gtU$ is a good representation, 
we have 
\begin{align}\label{eq:basis}
	\begin{aligned}
	\bm{d}_{i2}&=\frac{t_i-a_{i0}}{a_{i1}}\bm{d}_{i1} & \quad & \mbox{over $U_i$}, \\
%	\bm{d}_{i1}&=\frac{t_i+a_{i0}}{a_{i2}}\bm{d}_{i2} & \quad & \mbox{over $U_{i2}:=U_i\cap\{a_{i2}\ne 0\}$},  \\ 
%	\bm{d}_{i2}&=\frac{t_i+a_{i0}'}{a_{i3}}\bm{d}_{i3} & \quad & \mbox{over $U_{i3}:=U_i\cap\{a_{i3}\ne0\}$}.
	\end{aligned}
\end{align}
Therefore $\Ima(\eta)\subset\CC(X)$ is a line bundle over $X$. 
{\rev 
By the construction of $\eta$, it is clear that $\Upsilon([\Ima(\eta)])=[(\mcM, M)]$. 
Hence $\Upsilon$ is surjective. 
}
%The isomorphism class of $\Ima(\eta)$ does not depend on the choice of $\bm{d}_{i_01}, \bm{d}_{i_02}$ since the $\mcO_X$-algebra structure of $\Ima(\eta)$ does not depend on the choice of $\bm{d}_{i_01}, \bm{d}_{i_02}$. %(cf. Corollary~\ref{cor:corr}~\ref{cor:corr_transformation} below). 
%Let $\mcL_{(\mcM,M)}$ denote the line bundle isomorphic to $\Ima(\eta)$. 
%
%If $(\mcN,N)$ is an admissible pair which is equivalent to $(\mcM,M)$, then, for an isomorphism $\Psi:\mcM\to\mcN$ with $\Psi\circ M=N\circ\Psi(-L)$, we have $\Ima(\eta\circ\Psi^{-1})=\Ima(\eta)$ and $N$ is compatible with the multiplication by $t$ under $\eta\circ\Psi^{-1}$; 
%hence $\mcL_{(\mcM,M)}\cong\mcL_{(\mcN,N)}$. 
%Thus we obtain the map $\Upsilon:\ad_\phi(Y)\to\pic(X)$. 
%
%Since $X$ and $Y$ are smooth, $\phi$ is flat. 
%Hence $\phi_\ast\mcL$ is a $2$-bundle on $Y$ for a line bundle $\mcL$ on $X$. 
%A line bundle $\mcL$ on $X$ can be regarded as a subbundle of $\CC(X)$ naturally, and the multiplication by $t\in\HH^0(X,\mcO_X(R_\phi))$ gives a morphism $M_\mcL:\phi_\ast\mcL(-L)\to\phi_\ast\mcL$ with $M_\mcL^2=F{\cdot}\id_{\phi_\ast\mcL}$. 
%Therefore $(\phi_\ast\mcL,M_\mcL)$ is an admissible pair for $\phi$, and we have $\mcL_{(\phi_\ast\mcL,M_\mcL)}\cong\mcL$. 
%It is clear that $\phi_\ast\mcL_{(\mcM,M)}\cong\mcM$. 
%Hence $\Upsilon$ is one-to-one correspondence. 
\end{proof}

\begin{defin}
Let $(\mcM,M)$ be an admissible pair for $\phi:X\to Y$, and let $(\{G_{ij}\},\{M_i\})_\gtU$ be a representation of $(\mcM,M)$. 
For a line bundle $\mcL$ on $X$, 
we say that \textit{$\mcL$ is associated to $(\mcM,M)$} (or $(\{G_{ij}\},\{M_i\})_\gtU$) if 
{\rev 
$\Upsilon([\mcL])=[(\mcM,M)]$. 
A line bundle on $X$ associated to $(\mcM,M)$ is denoted by $\mcL_{(\mcM,M)}$. 
}
\end{defin}

\begin{cor}\label{cor:corr}
	Let $(\mcM,M)$ be an admissible pair for $\phi:X\to Y$, and let $(\{G_{ij}\},\{M_i\})_\gtU$ be a good representation of $(\mcM,M)$, where $G_{ij}$ and $M_i$ are defined as {\rm (\ref{eq:Gij})} and {\rm (\ref{eq:Mi})}, respectively. 
%	{\rev
%	Assume that a line bundle $\mcL$ on $X$ is associated to $(\mcM,M)$. 
%	}
	Put $\mcL:=\mcL_{(\mcM,M)}$. 
	Then the following statements hold:
	\begin{enumerate}[label=\rm (\roman{enumi})]
	\item\label{cor:corr_transformation} For $V_i:=\phi^{-1}(U_i)$, there are isomorphisms $\tilde{\varphi_i}:\mcL|_{V_i}\overset{\sim}{\to}\mcO_{V_i}$ $(i\in I)$ such that $\tilde{\varphi}_i\circ\tilde{\varphi}_j^{-1}:\mcO_{V_j}\to\mcO_{V_i}$ on $V_i\cap V_j$ is defined by 
	\begin{align}
	\tilde{\varphi}_i\circ\tilde{\varphi}_j^{-1}(1)=g_{ij,11}+g_{ij,21}\,\dfrac{t_i-a_{i0}}{a_{i1}}.
	\label{eq:mcLi}
%	\begin{array}{cccc} 
%		\tilde{\phi}_i\circ\tilde{\varphi}_j^{-1}:&
%		\mcO_{V_j}|_{V_i\cap V_j}&
%		\to &
%		\mcO_{V_i}|_{V_i\cap V_j} 
%	\\
%		&
%		\rotatebox{90}{$\in$}
%		& &
%		\rotatebox{90}{$\in$}
%	\\
%		&
%		1
%		& \mapsto &
%		g_{ij,11}+g_{ij,21}\,\dfrac{t_i-a_{i0}}{a_{i1}}
%	\end{array}
	\end{align}
	\item\label{cor:corr_morphism} A natural isomorphism $\upsilon:\mcM\to\phi_\ast\mcL$ of $\mcO_Y$-modules is locally given on $U_i$ by 
	\begin{align*}
		\tilde{\varphi}_i\circ\upsilon|_{U_i}\circ\varphi_i^{-1}\begin{pmatrix} 1 \\ 0 \end{pmatrix}&=1,
		& 
		\tilde{\varphi}_i\circ\upsilon|_{U_i}\circ\varphi_i^{-1}\begin{pmatrix} 0 \\ 1 \end{pmatrix}&=\frac{t_i-a_{i0}}{a_{i1}}.
	\end{align*}
	\end{enumerate}
\end{cor}
\begin{proof}
Let $\eta:\mcM\to\CC(X)$ and $\bm{d}_{i1}, \bm{d}_{i2}$ be as in the proof of Proposition~\ref{prop:corr}. 
By the proof of Proposition~\ref{prop:corr}, $\phi_\ast\mcL$ is  isomorphic to the image of $\eta$, and $\eta$ gives a natural isomorphism $\upsilon:\mcM\to\phi_\ast\mcL$.  

%\ref{cor:corr_transformation} 
Since $\mcL|_{V_i}$ is generated by $\bm{d}_{i1}$ in $\CC(X)$ by (\ref{eq:basis}), an isomorphism $\tilde{\varphi}_i:\mcL|_{V_i}\to\mcO_{V_i}$ is defined by $\tilde{\varphi}_i(\bm{d}_{i1})=1$. 
Then $\tilde{\varphi}_i\circ\tilde{\varphi}_j^{-1}(1)=\tilde{\varphi}_i(\bm{d}_{j1})=\bm{d}_{j1}/\bm{d}_{i1}$ on $V_i\cap V_j$. 
By (\ref{eq:transform}) and (\ref{eq:basis}), we obtain
\begin{align*}
	\frac{\bm{d}_{j1}}{\bm{d}_{i1}}=g_{ij,11}+g_{ij,21}\,\frac{t_i-a_{i0}}{a_{i 1}}. 
\end{align*}
Thus \ref{cor:corr_transformation} holds. 
For \ref{cor:corr_morphism},
we have 
\begin{align*}
		\tilde{\varphi}_i\circ\upsilon|_{U_i}\circ\varphi_i^{-1}\begin{pmatrix} 1 \\ 0 \end{pmatrix}
		&=
		\tilde{\varphi}_i(\bm{d}_{i1})=1,
		\\
		\tilde{\varphi}_i\circ\upsilon|_{U_i}\circ\varphi_i^{-1}\begin{pmatrix} 0 \\ 1 \end{pmatrix}&=\tilde{\varphi}_i(\bm{d}_{i2})=\frac{\bm{d}_{i2}}{\bm{d}_{i1}}=\frac{t_i-a_{i0}}{a_{i1}}.
\end{align*}
This completes the proof. 
\end{proof}

\begin{ex}\label{ex:rami}
Let $\phi:X\to Y$ be a non-singular double cover. 
Let $\{U_i\}_{i\in I}$ be an affine open covering of $Y$ such that $\phi_\ast\mcO_X|_{U_i}\cong\mcO_{U_i}\oplus\mcO_{U_i} t_i$ for each $i\in I$. 
The ramification divisor $R_\phi$ is defined by $t=0$. 
Hence we can take 
\[ \bm{d}_{i1}:=\frac{1}{t_i}, \qquad \bm{d}_{i2}:=t_i\bm{d}_{i1}=1 \]
as local basis of $\phi_\ast\mcO_X(R_\phi)$ in $\CC(X)$ as an $\mcO_Y$-module. 
Then we have
\[ 
 \begin{pmatrix}
 	\bm{d}_{j1} & \bm{d}_{j2}
 \end{pmatrix}
 =
 \begin{pmatrix}
 	\bm{d}_{i1} & \bm{d}_{i2}
 \end{pmatrix}
 \begin{pmatrix}
 	\xi_{ij}^{-1} & 0 \\ 0 & 1
 \end{pmatrix}, 
 \quad\ 
 t_i
 \begin{pmatrix}
 	\bm{d}_{i1} & \bm{d}_{i2}
 \end{pmatrix}
 =
 \begin{pmatrix}
 	\bm{d}_{i1} & \bm{d}_{i2}
 \end{pmatrix}
 \begin{pmatrix}
 	0 & F_i \\ 1 & 0
 \end{pmatrix}.
\]
Therefore, $\mcO_X(R_\phi)$ is associated to the admissible pair $(\mcO_Y(L)\oplus\mcO_Y, M)$ for $\phi$, where $M$ is given by 
\[ M=
 \begin{pmatrix}
 	0 & F \\ 1 & 0
 \end{pmatrix}:\mcO_Y\oplus\mcO_Y(-L)\to\mcO_Y(L)\oplus\mcO_Y.
 \]
\end{ex}

We define a normal representation of an admissible pair. In Corollary~\ref{cor:normal} below, we will prove that any admissible pair has a normal representation. 

\begin{defin}
Let $(\mcM,M)$ is an admissible pair for a non-singular double cover $\phi:X\to Y$, and let $F\in\HH^0(Y,\mcO_Y(B_\phi))$ be a global section defining $B_\phi$. 
A representation $(\{G_{ij}\},\{M_i\})_{\gtU}$ is said to be \textit{normal} if 
\begin{align*}
	M_i=
	\begin{pmatrix}
		0 & F_i \\ 1 & 0
	\end{pmatrix}
\end{align*}
for each $i\in I$, where $F_i:=F|_{U_i}$. 
\end{defin}

It is clear that a normal representation is good. 
The following lemma gives a criterion for equivalence of two normal representations. 

\begin{lem}\label{lem:equivalence}
Let $(\mcM,M)$ and $(\mcN,N)$ be two admissible pairs for a non-singular double cover $\phi:X\to Y$, and let $(\{G_{ij}\},\{M_i\})_{\gtU}$ and $(\{H_{ij}\},\{N_i\})_{\gtU}$ be normal representations of $(\mcM,M)$ and $(\mcN,N)$, respectively,
%such that 
%\begin{align}
%	M_i=N_i=
%	\begin{pmatrix}
%	0 & F_i \\ 1 & 0
%	\end{pmatrix}
%	\quad (i\in I), 
%	\label{eq:normal}
%\end{align}
where $\gtU=\{U_i\}_{i\in I}$ is an affine open covering of $Y$. 
Then $(\mcM,M)$ and $(\mcN,N)$ are equivalent if and only if there exist $\alpha_i,\beta_i\in \Gamma(U_i,\mcO_Y)$ for any $i\in I$ such that $\alpha_i^2-\beta_i^2 F_i\in\Gamma(U_i,\mcO_Y^\times)$ and 
\begin{align*}
	\begin{pmatrix}
		\alpha_i & \beta_i F_i \\ \beta_i & \alpha_i
	\end{pmatrix}
	G_{ij}
	&=
	H_{ij}
	\begin{pmatrix}
		\alpha_j & \beta_j F_j \\ \beta_j & \alpha_j
	\end{pmatrix}
	\quad (i,j \in I). 
\end{align*}
\end{lem}
\begin{proof}
Suppose that $(\mcM,M)$ and $(\mcN,N)$ are equivalent. 
Then there exists an isomorphism $\Psi:\mcM\to\mcN$ such that $\Psi{\circ} M=N{\circ}\Psi(-L)$. 
For each $i\in I$, $\Psi|_{U_i}$ is represented by $W_i\in\GL(2,\mcO_{U_i})$
\begin{align*}
	W_i:=
	\begin{pmatrix}
		\alpha_i & \gamma_i \\ \beta_i & \delta_i
	\end{pmatrix}
	: \mcO_{U_i}\oplus\mcO_{U_i}\cong \mcM|_{U_i} \overset{\Psi}{\to} \mcN|_{U_i} \cong \mcO_{U_i}\oplus\mcO_{U_i}
\end{align*}
satisfying $W_iM_i=N_iW_i$ and $\det(W_i)\in\Gamma(U_i,\mcO_Y^\times)$. 
By the definition of normal representations, we obtain 
\begin{align}\label{eq:Wi}
	W_i=
	\begin{pmatrix}
		\alpha_i & \beta_i F_i \\ \beta_i & \alpha_i
	\end{pmatrix} 
\end{align}
with $\alpha_i^2-\beta_i^2 F_i\in\Gamma(U_i,\mcO_Y^\times)$. 
Moreover $W_i$ ($i\in I$) satisfy $W_iG_{ij}=H_{ij}W_j$. 

Conversely, suppose that there exist $W_i\in\GL(U_i,\mcO_Y^\times)$ ($i\in I$) defined by (\ref{eq:Wi}) such that $W_iG_{ij}=H_{ij}W_j$ for each $i,j\in I$. 
It is clear that $W_iM_i=N_iW_i$. 
Therefore, $W_i$ ($i\in I$) defines an isomorphism $\Psi:\mcM\to\mcN$ with $\Psi{\circ}M=N{\circ}\Psi(-L)$, and $(\mcM,M)$ and $(\mcN,N)$ are equivalent. 
\end{proof}

\section{Group structure of $\ad_\phi(Y)$}\label{sec:group_structure}

Let $\phi:X\to Y$ be a non-singular double cover. 
The set $\ad_\phi(Y)$ has a group structure induced by one of $\pic(X)$ through {\rev the bijection} $\Upsilon$ in Proposition~\ref{prop:corr}. 
We next investigate the group structure of $\ad_\phi(Y)$. 
Let $\big(\mcM^{(k)},M^{(k)}\big)$ be an admissible pair for a non-singular double cover $\phi:X\to Y$ for each $k=1,\dots,m$. 
%We compute the admissible pair corresponding to $\mcL_{(\mcM^{(1)},M^{(1)})}^{\otimes n_1}\otimes\dots\otimes\mcL_{(\mcM^{(m)},M^{(m)})}^{\otimes n_m}$ for non-negative integers $n_1,\dots,n_m\geq 0$. 
Let $\big(\{G_{ij}^{(k)}\},\{M_i^{(k)}\}\big)_\gtU$ be a good representation of $\big(\mcM^{(k)},M^{(k)}\big)$ for each $k=1,\dots,m$, where 
\begin{align*}
	G_{ij}^{(k)}&=
	\begin{pmatrix}
		g_{ij,11}^{(k)} & g_{ij,12}^{(k)} \\ g_{ij,21}^{(k)} & g_{ij,22}^{(k)}
	\end{pmatrix}, 
	&
	M_i^{(k)}&=
	\begin{pmatrix}
		a_{i0}^{(k)} & a_{i2}^{(k)} \\ a_{i1}^{(k)} & -a_{i0}^{(k)}
	\end{pmatrix}
	\quad (k=1,\dots,m). 
%	\label{eq:Gij^k_Mi^k}
\end{align*}
Let $\big(\{G_{ij}^{(0)}\},\{M_i^{(0)}\}\big)_{\gtU}$ be the good representation of {\rev $\Upsilon([\mcO_X])$}, where %$G_{ij}^{(0)}$ and $M_i^{(0)}$ are the matrices $H_{ij}$ and $N_i$ in (\ref{eq:trivial}). 
\begin{align*}
	G_{ij}^{(0)}&=
	\begin{pmatrix}
		1 & 0 \\ 0 & \xi_{ij}
	\end{pmatrix},
	&
	M_i^{(0)}&=
	\begin{pmatrix}
		0 & F_i \\ 1 & 0
	\end{pmatrix}. 
%	\label{eq:trivial}
\end{align*}
For $2\times 2$ matrices $A_1,\dots,A_m$, put $\prod_{k=1}^mA_k:=A_1A_2\dots A_m$. 
Then the following theorem holds, and {\rev proves} Theorem~\ref{thm:group_law0}. 

\begin{thm}\label{thm:group_law}
With the above notation, put 
\begin{align}
	K_{ij}^{(k)+} &:= 
%	g_{ij,11}^{(k)}E+g_{ij,21}^{(k)}R_{(M_i^{(k)},M_i^{(0)})}
%	\\
%	&=
	\frac{1}{a_{i1}^{(k)}}\bigg( \left( a_{i1}^{(k)}g_{ij,11}^{(k)} - a_{i0}^{(k)}g_{ij,21}^{(k)} \right)E + g_{ij,21}^{(k)}M_i^{(0)} \bigg),
	\label{eq:Kij^k}
\\
	K_{ij}^{(k)-}&:=
%	\frac{\xi_{ij}}{\det(G_{ij})} \left(g_{ij,11}^{(k)} E+g_{ij,21}^{(k)}R_{(-M_i^{(k)},M_i^{(0)})}\right)
%	\\
%	&=
	\frac{\xi_{ij}}{a_{i1}^{(k)}\det(G_{ij})}\bigg( \left( a_{i1}^{(k)}g_{ij,11}^{(k)}-a_{i0}^{(k)}g_{ij,21}^{(k)} \right)E -g_{ij,21}^{(k)}M_i^{(0)} \bigg)
	\label{eq:-Kij^k}
\end{align}
for each $k=1,\dots,m$. 
Let $n_1,\dots,n_m$ be $m$ integers, and let $[n]$ be the list $[n_1,\dots,n_m]$. 
Then $\mcL^{[n]}:=\mcL_{(\mcM^{(1)},M^{(1)})}^{\otimes n_1}\otimes\dots\otimes\mcL_{(\mcM^{(m)},M^{(m)})}^{\otimes n_m}$ is associated to the normal representation $\big(\{G_{ij}^{[n]}\},\{M_i^{[n]}\}\big)_\gtU$, where 
\begin{align*}
	G_{ij}^{[n]}&:=
	\prod_{k=1}^{m}\left(K_{ij}^{(k)}(n_k)\right)^{|n_k|} %\left(\widetilde{K}_{ij}^{(k)}\right)^{\frac{-n_k+|n_k|}{2}}\right)
	G_{ij}^{(0)},
	&
	M_{i}^{[n]}&:=M_i^{(0)},
\end{align*}
where $K_{ij}^{(k)}(n_k):=K_{ij}^{(k)+}$ if $n_k\geq0$, and $K_{ij}^{(k)}(n_k):=K_{ij}^{(k)-}$ otherwise. 
\end{thm}

To prove Theorem~\ref{thm:group_law}, we first compute the admissible pair {\rev $\Upsilon([\mcL^{[1,1]}])$} in the case where $m=2$ and $ {\rev [n]}=[1,1]$. 
For simplicity, put $(\mcM,M):=(\mcM^{(1)},M^{(1)})$ and $(\mcN,N):=(\mcM^{(2)},M^{(2)})$. 
%an admissible pair $\big(\{G_{ij}^\sim\},\{M_i^\sim\}\big)_\gtU$ corresponding to $\mcL_{(\mcM,M)}\otimes\mcL_{(\mcN,N)}$ for two admissible pairs $(\mcM,M)$ and $(\mcN,N)$. 
Let $(\{G_{ij}\},\{M_i\})_{\gtU}$ and $(\{H_{ij}\},\{N_i\})_{\gtU}$ be good representations of $(\mcM,M)$ and $(\mcN,N)$, respectively, for an affine open covering $\gtU=\{U_i\}_{i\in I}$ of $Y$. 
Write  
\begin{align*}
	G_{ij}&:=
	\begin{pmatrix}
		g_{ij,11} & g_{ij,12} \\ g_{ij,21} & g_{ij,22}
	\end{pmatrix}, 
	&
	M_i&:=
	\begin{pmatrix}
		a_{i0} & a_{i2} \\ a_{i1} & -a_{i0}
	\end{pmatrix}, 
	&
	N_i&:=
	\begin{pmatrix}
		b_{i0} & b_{i2} \\ b_{i1} & -b_{i0}
	\end{pmatrix}.
\end{align*}

\begin{prop}\label{prop:tensor}
	With the above notation, 
	$\mcL_{(\mcM,M)}\otimes\mcL_{(\mcN,N)}$ is associated to 
%	the admissible pair $\Upsilon^{-1}(\mcL^{[1,1]})$ is represented by 
%	\begin{align}
%		\mcL_{(\mcM,M)}\otimes\mcL_{(\mcN,N)}\cong\mcL_{(\mcM^\sim,M^\sim)}
%	\end{align}
%	holds, 
%	where $(\mcM^\sim,M^\sim)$ is represented by 
the good representation $(\{G_{ij}^\sim\},\{M_{i}^\sim\})_{\gtU}$, where 
	\begin{align*}
		G_{ij}^\sim &:=\frac{1}{a_{i1}}\bigg(\left(a_{i1}g_{ij,11}-a_{i0}g_{ij,21}\right)E+g_{ij,21}N_i\bigg)H_{ij},
		&
		M_i^\sim &:=N_i 
%		&
%		(i,j\in I).
%	\label{eq:tensor}
	\end{align*}
\end{prop}
\begin{proof}
Let $\eta_M:\mcM\to\CC(X)$ and $\eta_N:\mcN\to\CC(X)$ be embeddings such that $t_i\eta_M=\eta_M\circ M_i$ and $t_i\eta_N=\eta_N\circ N_i$ on $U_i$ as the proof of Proposition~\ref{prop:corr}, and put 
\begin{align*}
	\bm{d}_{i1}&:= \eta_M \begin{pmatrix} 1 \\ 0\end{pmatrix}, 
	&
	\bm{d}_{i2}&:= \eta_M \begin{pmatrix} 0 \\ 1\end{pmatrix}, 
	&
	\bm{e}_{i1}&:= \eta_N \begin{pmatrix} 1 \\ 0 \end{pmatrix}, 
	&
	\bm{e}_{i2}&:= \eta_N \begin{pmatrix} 0 \\ 1 \end{pmatrix} 
\end{align*}
on $U_i$ for $i\in I$, where $\mcM_i:=\mcM|_{U_i}$ and $\mcN_i:=\mcN|_{U_i}$ are identified with $\mcO_{U_i}^{\oplus 2}$. 
Put $V_i:=\phi^{-1}(U_i)$ for $i\in I$. 
Then $\mcL_{(\mcM,M)}$ and $\mcL_{(\mcN,N)}$ are generated by $\bm{d}_{i 1}$ and $\bm{e}_{i1}$ on $V_i$, respectively, as $\mcO_X$-modules in $\CC(X)$.  
The line bundle $\mcL_{(\mcM,M)}\otimes\mcL_{(\mcN,N)}$ is locally generated by $\bm{d}_{i1}\bm{e}_{i1}$ on $V_{i}$ as an $\mcO_X$-module in $\CC(X)$. 
Thus $\mcE:=\phi_\ast\big(\mcL_{(\mcM,M)}\otimes\mcL_{(\mcN,N)}\big)$ is locally generated by $\bm{d}_{i 1}\bm{e}_{i1}$ and $t_i\bm{d}_{i 1}\bm{e}_{i1}$ on $U_{i}$ as an $\mcO_{Y}$-submodule of $\CC(X)$. 
By $t_i\eta_M=\eta_M\circ M_i$ and $t_i\eta_N=\eta_N\circ N_i$, we have 
\begin{align}\label{eq:11_12}
	\begin{aligned}
	t_i\bm{d}_{i1}\bm{e}_{i1}
	&=a_{i0}\bm{d}_{i1}\bm{e}_{i1}+a_{i1}\bm{d}_{i2}\bm{e}_{i1}
	=b_{i0}\bm{d}_{i1}\bm{e}_{i1}+b_{i1}\bm{d}_{i1}\bm{e}_{i2}, 
	\\
	t_i\bm{d}_{i1}\bm{e}_{i2}
	&=a_{i0}\bm{d}_{i1}\bm{e}_{i2}+a_{i1}\bm{d}_{i2}\bm{e}_{i2}
	=b_{i2}\bm{d}_{i1}\bm{e}_{i1}-b_{i0}\bm{d}_{i1}\bm{e}_{i2}. 
	\end{aligned}
\end{align}
Hence we obtain 
\begin{align*}
	\bm{d}_{i1}\bm{e}_{i2}&=\frac{t_i-b_{i0}}{b_{i1}}\bm{d}_{i1}\bm{e}_{i1}.
%	&
%	\bm{d}_{i1}\bm{e}_{i1}&=\frac{t_i+b_{i0}}{b_{i2}}\bm{d}_{i1}\bm{e}_{i2}.
\end{align*}
Thus $\mcE$ is locally generated by $\bm{d}_{i1}\bm{e}_{i1}, \bm{d}_{i1}\bm{e}_{i2}$ on $U_{i}$ for any $i\in I$. 
Let 
$\bm{d}_{i k}^\sim:=\bm{d}_{i1}\bm{e}_{ik}$ for $i\in I$ and $k=1,2$. 
By (\ref{eq:11_12}), the equation
\begin{align*}%\label{eq:iota_base_change}
	\begin{pmatrix}
		\bm{d}_{i 1}^\sim & \bm{d}_{i 2}^\sim
	\end{pmatrix}
	R_{(M_i,N_i)}
	&=
	\begin{pmatrix}
		\bm{d}_{i2}\bm{e}_{i1} & \bm{d}_{i 2}\bm{e}_{i2}
	\end{pmatrix}
	\qquad \left( R_{(M_i,N_i)}:=\frac{1}{a_{i1}}(-a_{i0}E+N_i) \right)
\end{align*}
holds for $i\in I$. 
We construct the set of matrices $\{G_{ij}^\sim\}_{i,j\in I}$ with $\begin{pmatrix}\bm{d}_{j 1}^\sim & \bm{d}_{j 2}^\sim \end{pmatrix}=\begin{pmatrix}\bm{d}_{i 1}^\sim & \bm{d}_{i 2}^\sim \end{pmatrix}G_{ij}^\sim$ for $i,j\in I$. 
For $i,j\in I$, 
\begin{align*}
%	\begin{aligned}
	&
	\begin{pmatrix}
		\bm{d}_{j1}^\sim & \bm{d}_{j2}^\sim
	\end{pmatrix}
	\begin{pmatrix}
		E &	R_{(M_j,N_j)}
	\end{pmatrix}
	\\
	=&
	\begin{pmatrix}
		\bm{d}_{j1} & \bm{d}_{j2}
	\end{pmatrix}
	\begin{pmatrix}
		\bm{e}_{j1} & \bm{e}_{j2} & 0 & 0 \\ 
		0 & 0 & \bm{e}_{j1} & \bm{e}_{j2}
	\end{pmatrix}
	\\
	=&
	\begin{pmatrix}
		\bm{d}_{i1} & \bm{d}_{i2}
	\end{pmatrix}
	G_{ij}
	\begin{pmatrix}
		\bm{e}_{i1} & \bm{e}_{i2} & 0 & 0 \\ 
		0 & 0 & \bm{e}_{i1} & \bm{e}_{i2}
	\end{pmatrix}
	\begin{pmatrix}
		H_{ij} & O \\ O & H_{ij}
	\end{pmatrix}
	\\
	=&
	\begin{pmatrix}
		\bm{d}_{i1}^\sim & \bm{d}_{i2}^\sim
	\end{pmatrix}
	\begin{pmatrix}
		E &	R_{(M_i,N_i)}
	\end{pmatrix}
	\begin{pmatrix}
		g_{ij,11}H_{ij} & g_{ij,12}H_{ij} \\ g_{ij,21}H_{ij} & g_{ij,22}H_{ij}
	\end{pmatrix}.
%	\end{aligned}
\end{align*}
Therefore we obtain 
\begin{align*}
	\begin{pmatrix}
		\bm{d}_{j1}^\sim & \bm{d}_{j2}^\sim
	\end{pmatrix}
	&=
	\begin{pmatrix}
		\bm{d}_{i1}^\sim & \bm{d}_{i2}^\sim
	\end{pmatrix}
	\begin{pmatrix}
		E & R_{(M_i,N_i)}
	\end{pmatrix}
	\begin{pmatrix}
		g_{ij,11}H_{ij} \\ g_{ij,21}H_{ij}
	\end{pmatrix}. 
\end{align*}
Note that $G_{ij}^\sim=(g_{ij,11}E+g_{ij,21}R_{(M_i,N_i)})H_{ij}$. 
Hence we obtain $\begin{pmatrix} \bm{d}_{j1}^\sim & \bm{d}_{j2}^\sim \end{pmatrix}=\begin{pmatrix} \bm{d}_{i1}^\sim & \bm{d}_{i2}^\sim \end{pmatrix}G_{ij}^\sim$. 
By the definition of $R_{(M_i,N_i)}$, $G_{ij}^\sim$ is defined on $U_i\cap U_j$. 
Note that $G_{ii}^\sim=E$ and 
	$\begin{pmatrix}\bm{d}_{i1}^\sim & \bm{d}_{i2}^\sim \end{pmatrix}G_{ij}^\sim G_{jk}^{\sim}
	=
	\begin{pmatrix}\bm{d}_{k1}^\sim & \bm{d}_{k2}^\sim \end{pmatrix}$ for any $i,j,k\in I$. 
Since $\bm{d}_{j1},\bm{d}_{j2}$ are linearly independent over $\CC(Y)$, we have $G_{ij}^\sim G_{jk}^\sim=G_{ik}^\sim$, and $G_{ij}^\sim\in\GL(2,\mcO_{U_i\cap U_j})$. 
Hence $\{G_{ij}^\sim\}_{i,j\in I}$ defines $\mcE$. 
By (\ref{eq:11_12}), it is clear that 
\begin{align*}
t_i
\begin{pmatrix}
	\bm{d}_{i1}^\sim & \bm{d}_{i2}^\sim
\end{pmatrix}
=
\begin{pmatrix}
	\bm{d}_{i1}^\sim & \bm{d}_{i2}^\sim
\end{pmatrix}
N_i. 
 \end{align*}
Hence, putting $M_i^\sim:=N_i$, the pair $(\{G_{ij}^\sim\},\{M_i^\sim\})_\gtU$ is an admissible pair corresponding to $\mcL_{(\mcM,M)}\otimes\mcL_{(\mcN,N)}$. 
\end{proof}

We can prove Theorem~\ref{thm:group_law} in the case where $n_k\geq 0$ for $k=1,\dots,m$ by Proposition~\ref{prop:tensor}. 

\begin{lem}\label{lem:n-tensor}
If $n_k\geq 0$ for any $k=1,\dots,m$, then the line bundle $\mcL^{[n]}$ is associated to the normal representation $\big(\{G_{ij}^{[n]}\},\{M_i^{[n]}\}\big)_\gtU$. 
\end{lem}
\begin{proof}
Since $n_k\geq 0$, we have $K_{ij}^{(k)}(n_k)=K_{ij}^{(k)+}$ for each $k=1,\dots,m$. 
Thus, in this case, 
\begin{align*}
	G_{ij}^{[n]}&:=
	\prod_{k=1}^m\left(K_{ij}^{(k)+}\right)^{n_k}G_{ij}^{(0)}, 
	&
	M_{i}^{[n]}&:=M_i^{(0)}. 
\end{align*}
We prove this lemma by induction on $n_1+\dots+n_m$. 
In the case of $n_1+\dots+n_m=0$, the assertion holds since $n_1=\dots=n_m=0$. 
Suppose that $n_1+\dots+n_m>0$. 
Put $k_0:=\min\{k\mid n_k>0\}$. 
By the assumption of the induction, $\mcL^{[n']}$ is associated to $\big( \{G_{ij}^{[n']}\}, \{M_i^{[n']}\} \big)_\gtU$, 
%\begin{align}
%	\begin{aligned}
%	G_{ij}^{[n']}&=\left(K_{ij}^{(k_0)}\right)^{n_{k_0}-1}\left(K_{ij}^{(k_0+1)}\right)^{n_{k_0+1}}\dots\left(K_{ij}^{(m)}\right)^{n_m}G_{ij}^{(0)}, 
%	\\
%	M_i^{[n']}&=M_i^{(0)}, 
%	\end{aligned}
%\end{align}
where $[n']=[0,\dots,0,n_{k_0}-1,n_{k_0+1},\dots,n_m]$. 
Since $\mcL^{[n]}=\mcL_{(\mcM^{(k_0)},M^{(k_0)})}\otimes\mcL^{[n']}$ and $M_i^{[n']}=M_i^{(0)}$, 
%$R_{(M_{i}^{(k_0)},M_i^{[n']})}=R_{(M_{i}^{(k_0)},M_i^{(0)})}$, 
$\mcL^{[n]}$ is associated to $\big( \{\mcM^{[n]}\}, \{M^{[n]}\} \big)_\gtU$ by Proposition~\ref{prop:tensor}. 
%\begin{align}
%	G_{ij}^{[n]}&=K_{ij}^{(k_0)}G_{ij}^{[n']}, 
%	&
%	M_i^{[n]}&=M_i^{[n']}=M_i^{(0)}. 
%\end{align}
%Hence the proof is complete. 
\end{proof}

\begin{cor}\label{cor:normal}
For any admissible pair $(\mcM,M)$ for a non-singular double cover $\phi$, there exists a normal representation of $(\mcM,M)$. 
\end{cor}
\begin{proof}
In the case where $m=1$ and $\big(\mcM^{(1)},M^{(1)}\big):=(\mcM,M)$, Lemma~\ref{lem:n-tensor} shows that $\mcL^{[1]}=\mcL_{(\mcM,M)}$ is associated to $\big( \{G_{ij}^{[1]}\}, \{M_i^{[1]}\} \big)_{\gtU}$ which is normal. 
\end{proof}

\begin{prop}\label{prop:inverse}
Let $(\mcM,M)$ be an admissible pair for a non-singular double cover $\phi:X\to Y$, and let $(\{G_{ij}\},\{M_i\})_\gtU$ be a good representation of $(\mcM,M)$.  
Let $\iota:X\to X$ be the covering transformation of $\phi$. 
Then the followings hold:
\begin{align}
	\iota^\ast\mcL_{(\mcM,M)} &\cong \mcL_{(\mcM,-M)}, 
\label{eq:conjugate}\\
	\mcL_{(\mcM,M)}\otimes \iota^\ast\mcL_{(\mcM,M)} &\cong \phi^\ast\Big( (\det\mcM)\otimes\mcO_Y(L) \Big), 
\label{eq:inv_sheaf}\\
	\mcL_{(\mcM,M)}^{-1} &\cong \phi^\ast\Big( (\det\mcM)^{-1}\otimes\mcO_Y(-L) \Big)\otimes\mcL_{(\mcM,-M)}. 
\label{eq:inverse}
\end{align}
Moreover, $\mcL_{(\mcM,M)}^{-1}$ is associated to 
\begin{align*}
	\left(
	\left\{
	\frac{\xi_{ij}}{\det(G_{ij})}G_{ij}
	\right\}
	, 
	\left\{
	-M_i
	\right\}
	\right)_\gtU. 
\end{align*}
\end{prop}
\begin{proof}
Assume $G_{ij}$ and $M_i$ are defined as (\ref{eq:Gij}) and (\ref{eq:Mi}), respectively. 
For $a+bt_i\in\CC(X)$ with $a,b\in\CC(Y)$, put $\overline{a+bt_i}:=a-bt_i$. 
We regard $\mcL_{(\mcM,M)}$ as an $\mcO_X$-submodule in $\CC(X)$ as in the proof of Proposition~\ref{prop:tensor}. 
The line bundle $\iota^\ast\mcL_{(\mcM,M)}$ is generated by $\overline{\bm{d}}_{i1}$ on $V_{i}:=\phi^{-1}(U_i)\subset X$ by (\ref{eq:basis}). 
Since $G_{ij}\in\GL(2,\mcO_{U_i\cap U_j})$, we obtain
\begin{align*}
	\begin{pmatrix}
		\overline{\bm{d}}_{j1} & \overline{\bm{d}}_{j2}
	\end{pmatrix}
	=\overline{
	\begin{pmatrix}
		\bm{d}_{i1} & \bm{d}_{i2}
	\end{pmatrix}
	G_{ij}}
	=
	\begin{pmatrix}
		\overline{\bm{d}}_{i1} & \overline{\bm{d}}_{i2}
	\end{pmatrix}
	G_{ij}.
\end{align*}
Thus $\phi_\ast(\iota^\ast\mcL_{(\mcM,M)})\cong \mcM$. 
Similarly, we have 
\begin{align*}
	t_i
	\begin{pmatrix}
		\overline{\bm{d}}_{i1} & \overline{\bm{d}}_{i2}
	\end{pmatrix}
	=
	\overline{-t_i
	\begin{pmatrix}
		\bm{d}_{i1} & \bm{d}_{i2}
	\end{pmatrix}
	}=
	\begin{pmatrix}
		\overline{\bm{d}}_{i1} & \overline{\bm{d}}_{i2}
	\end{pmatrix}
	(-M_i).
\end{align*}
Hence (\ref{eq:conjugate}) holds. %$\iota^\ast\mcL_{(\mcM,M)}\cong\mcL_{(\mcM,-M)}$. 

For (\ref{eq:inv_sheaf}), we compute $\mcL_{(\mcM,M)}\otimes\mcL_{(\mcM,-M)}$ by using Lemma~\ref{lem:n-tensor} putting $(\mcM^{(1)},M^{(1)}):=(\mcM,M)$ and $(\mcM^{(2)},M^{(2)}):=(\mcM,-M)$. 
In this case, we have 
\begin{align*} 
%	\begin{aligned}
	K_{ij}^{(1)+} &=\frac{1}{a_{i1}}\begin{pmatrix} a_{i1}g_{ij,11}-a_{i0}g_{ij,21} & g_{ij,21}F_i \\ g_{ij,21} & a_{i1}g_{ij,11}-a_{i0}g_{ij,21} \end{pmatrix}, 
	\\[0.5em]
	K_{ij}^{(2)+} &=\frac{1}{a_{i1}}\begin{pmatrix} a_{i1}g_{ij,11}-a_{i0}g_{ij,21} & -g_{ij,21}F_i \\ -g_{ij,21} & a_{i1}g_{ij,11}-a_{i0}g_{ij,21} \end{pmatrix}.  
%	\end{aligned}
\end{align*}
Since $F_i=a_{i0}^2+a_{i1}a_{i2}$, by direct computation, we obtain 
\begin{align*}
	G_{ij}^{[1,1]}&=
	\frac{a_{i1} g_{ij,11}^2 -2 a_{i0} g_{ij,11}g_{ij,21}  -a_{i2}g_{ij,21}^2}{a_{i1}}
	\begin{pmatrix}
		1 & 0 \\ 0 & \xi_{ij}
	\end{pmatrix}
\end{align*}
From the $(2,1)$ entry of the equation $M_j=\xi_{ij}G_{ij}^{-1}M_{i}G_{ij}$, we have 
\begin{align*}
	a_{j1}\det(G_{ij})&=\Big(a_{i1}g_{ij,11}^2 -2a_{i0}g_{ij,11}g_{ij,21} -a_{i2}g_{ij,21}^2\Big)\xi_{ij}. 
\end{align*}
Hence we obtain 
\begin{align}
	G_{ij}^{[1,1]}&=\frac{a_{j1}}{a_{i1}}\frac{\det(G_{ij})}{\xi_{ij}}
	\begin{pmatrix}
		1 & 0 \\ 0 & \xi_{ij}
	\end{pmatrix}.
	\label{eq:tran_push}
\end{align}
Moreover, since 
\begin{align*}
	\frac{1}{a_{i1}}
%	\begin{pmatrix}
%		1 & 0 \\ 0 & 1
%	\end{pmatrix}
	\Bigg(
	\frac{\det(G_{ij})}{\xi_{ij}}
	\begin{pmatrix}
		1 & 0 \\ 0 & \xi_{ij}
	\end{pmatrix}
	\Bigg)
	&=
	\frac{1}{a_{j1}}
	G_{ij}^{[1,1]}
%	\left(
%	\frac{a_j}{a_i}
%	\frac{\det(G_{ij})}{\xi_{ij}}
%	\begin{pmatrix}
%		1 & 0 \\ 0 & \xi_{ij}
%	\end{pmatrix}
%%	\begin{pmatrix}
%%		1 & 0 \\ 0 & 1
%%	\end{pmatrix}
%	\right)
	, 
\end{align*}
(\ref{eq:inv_sheaf}) holds by Lemma~\ref{lem:equivalence}. %$\mcL_{(\mcM,M)}\otimes\iota^\ast\mcL_{(\mcM,M)}\cong\phi^\ast\left( \det(\mcM)\otimes\mcO_Y(L) \right)$ holds. 
Isomorphism (\ref{eq:inverse}) follows from (\ref{eq:conjugate}) and (\ref{eq:inv_sheaf}). 
Since $\phi^\ast((\det\mcM)^{-1}\otimes\mcO_Y(-L))$ is associated to 
\begin{align*}
	\left( \left\{ \frac{\xi_{ij}}{\det(G_{ij})}
	\begin{pmatrix}
		1 & 0 \\ 0 & \xi_{ij}
	\end{pmatrix}
	\right\}, \left\{ 
	\begin{pmatrix}
		0 & F_i \\ 1 & 0
	\end{pmatrix}
	\right\} \right)_\gtU, 
\end{align*}
the last assertion follows from (\ref{eq:inverse}) and Proposition~\ref{prop:tensor}. 
\end{proof}

%We prove Theorem~\ref{thm:group_law}. 

\begin{proof}[Proof of Theorem~\ref{thm:group_law}]
%Since 
%\begin{align}
%	\frac{n_k+|n_k|}{2}&=\left\{ \begin{array}{ll} n_k & (n_k\geq 0) \\ 0 & (n_k\leq 0) \end{array} \right.,
% 	&
% 	\frac{-n_k+|n_k|}{2}&=\left\{ \begin{array}{ll} 0 & (n_k\geq 0) \\ -n_k & (n_k\leq 0) \end{array} \right.,
%\end{align}
The assertion follows from Lemma~\ref{lem:n-tensor} and Proposition~\ref{prop:inverse}. 
\end{proof}

In the rest of this section, we see the correspondence between global sections of $\mcL_{(\mcM,M)}$ and ones of $(\det\mcM)\otimes\mcO_Y(L)$. 
Let $(\mcM,M)$ be an admissible pair for $\phi$ represented by a good representation $\big(\{G_{ij}\},\{M_i\}\big)_{\gtU}$ ($\gtU:=\{U_i\}_{i\in I}$), where $G_{ij}$ and $M_i$ are written as (\ref{eq:Gij}) and (\ref{eq:Mi}), respectively. 
Put $\mcL:=\mcL_{(\mcM,M)}$. 
Let $\varphi_i:\mcM|_{U_i}\overset{\sim}{\to}\mcO_{U_i}^{\oplus 2}$ and $\tilde{\varphi}_i:\mcL|_{V_i}\overset{\sim}{\to}\mcO_{V_i}$ be isomorphisms satisfying (\ref{eq:Mi}) and (\ref{eq:mcLi}), respectively. 

\begin{prop}\label{prop:splitting}
With the above notation, let $\tilde{f}$ be a global section of $\mcL$, and put $\tilde{f}_i:=\tilde{f}|_{V_i}$ for $i\in I$. 
If $\tilde{\varphi}_i(\tilde{f}_i)=x_i+y_it_i$ with $x_i,y_i\in\Gamma(U_i,\mcO_Y)$, 
then $h_i:=a_{i1}(x_i^2-y_i^2 F_i)\in\Gamma(U_i,\mcO_Y)$ $(i\in I)$ satisfy 
\begin{align}
	\frac{\det(G_{ij})}{\xi_{ij}}h_j&=h_i. 
	\label{eq:det_h}
\end{align}
In particular, $\{h_i\}_{i\in I}$ defines a global section $h$ of $(\det\mcM)\otimes\mcO_{Y}(L)$. 
\end{prop}
\begin{proof}
Let $\bm{d}_{i1}, \bm{d}_{i2}$ be local basis of $\phi_\ast\mcL$ in $\CC(X)$ as in the proof of Proposition~\ref{prop:tensor}. 
By Corollary~\ref{cor:corr}~\ref{cor:corr_morphism}, we have 
\begin{align*} 
	\tilde{\varphi}_i\circ\upsilon|_{U_i}\circ\varphi_i^{-1}\begin{pmatrix} x_i+a_{i0}y_i \\ a_{i1}y_i \end{pmatrix} =\tilde{f}_i. 
\end{align*}
Hence the sections $(x_i+a_{i0}y_i)\bm{d}_{i1}+a_{i1}y_i\bm{d}_{i2}$ ($i\in I$) are glued each other, and defines the global section $\upsilon^{-1}(\tilde{f})\in\HH^0(Y,\mcM)$. 
By (\ref{eq:basis}) and direct computation, we obtain 
\begin{align*}
%	\begin{aligned} 
	&\Big((x_i+a_{i0}y_i)\bm{d}_{i1}+a_{i1}y_i\bm{d}_{i2}\Big)\Big((x_i+a_{i0}y_i)\overline{\bm{d}}_{i1}+a_{i1}y_i\overline{\bm{d}}_{i2}\Big)
	=
	h_i\frac{\bm{d}_{i1}\overline{\bm{d}}_{i1}}{a_{i1}}, 
%	\end{aligned}
\end{align*}
and they define a global section $\tilde{f}(\iota^\ast \tilde{f})$ of $\phi^\ast\big(\det\mcM\otimes\mcO_Y(L)\big)$ by Proposition~\ref{prop:inverse}. 
By the proof of Proposition~\ref{prop:tensor} and (\ref{eq:tran_push}), we obtain 
\begin{align*}
	\frac{\det(G_{ij})}{\xi_{ij}}
	\begin{pmatrix}
		\dfrac{\bm{d}_{i1}\overline{\bm{d}}_{i1}}{a_{i1}} & \dfrac{\bm{d}_{i1}\overline{\bm{d}}_{i2}}{a_{i1}}
	\end{pmatrix}
	\begin{pmatrix}
		1 & 0 \\ 0 & \xi_{ij}
	\end{pmatrix}
	=
	\begin{pmatrix}
		\dfrac{\bm{d}_{j1}\overline{\bm{d}}_{j1}}{a_{j1}} & \dfrac{\bm{d}_{j1}\overline{\bm{d}}_{j2}}{a_{j1}}
	\end{pmatrix}
\end{align*}
This implies that $\{h_i\}$ satisfies (\ref{eq:det_h}), and defines $h\in\HH^0\big(Y,(\det\mcM)\otimes\mcO_Y(L)\big)$. 
%Since $\phi_\ast\mcO_X\cong\mcO_Y\oplus\mcO_Y(-L)$ and $\phi^\ast(\det\mcM)(L)\cong\mcL_{(\mcM,M)}\otimes\iota^\ast\mcL_{\mcM,M}$, we have $f(\iota^\ast f)=\phi^\ast g$. 
\end{proof}

The following corollary is an interpretation of Proposition~\ref{prop:splitting} in terms of divisors. 

\begin{cor}\label{cor:splitting}
	With the same notation of Proposition~\ref{prop:splitting}, 
	let $D$ be an effective divisor on $Y$ with $\mcO_Y(D)\cong(\det\mcM)\otimes\mcO_Y(L)$. 
	If there exists an effective divisor $D^+$ on $X$ with $\mcO_{X}(D^+)\cong\mcL$ such that $\phi^\ast(D)=D^++\iota^\ast(D^+)$, then $D$ is locally defined by $a_{i1}(x_i^2-y_i^2F_i)=0$ for some $x_i,y_i\in\Gamma(U_i,\mcO_Y)$ on each $U_i$. 
\end{cor}
\begin{proof}
Let $\tilde{f}\in\HH^0(X,\mcL)$ be a section defining $D^+$, and put $\tilde{f}_i:=\tilde{f}|_{V_i}$. 
Let $x_i,y_i\in\Gamma(U_i,\mcO_Y)$ be sections such that $\tilde{\varphi}_i(\tilde{f}_i)=x_i+y_it_i$. 
By Proposition~\ref{prop:splitting}, $h_i:=a_{i1}(x_i^2-y_i^2F_i)$ ($i\in I$) define a section $h\in\HH^0\big(Y,(\det\mcM)\otimes\mcO_Y(L)\big)$. 
Moreover $h=0$ defines the image $\phi(D^+)=D$. 
Therefore $D$ is locally defined by $a_{i1}(x_i^2-y_i^2F_i)=0$. 
\end{proof}

\begin{rem}
Corollary~\ref{cor:splitting} gives a condition for splitting of $\phi^\ast C$. 
However, it is difficult to represent a local equation of a divisor as $x_i^2-y_i^2 F_i=0$. 
For example, there exists a plane curve $C\subset\PP^2$ of degree $6$  defined locally by $s_4^2-s_3^2F_U=0$ on an affine open $U\subset\PP^2$ for $s_i,F_U\in\CC[x,y]$ with $\deg s_i=i$ and $\deg F=2$, i.e., the higher terms of $s_4^2$ and $s_3^2F$ are canceled in $s_4^2-s_3^2F_U$ (see Example~\ref{ex:global_section} below). 
%For such plane curves $C$, it is still difficult to determine if $\phi^\ast C$ is split or not. 
\end{rem}

\section{A subgroup of $\Pic(X)$}\label{sec:subgroup}

We have seen the correspondence between admissible pairs for a non-singular double cover $\phi:X\to Y$ and line bundles on $X$. 
Hence it is effective for understanding $\pic(X)$ to study $\ad_\phi(Y)$. 
However, it seems difficult to find a morphism $M:\mcM(-L)\to\mcM$ satisfying $M^2=F\cdot\id_\mcM$ for a general $2$-bundle $\mcM$ on $Y$. 
In the case where $\mcM\cong\mcO_Y(D_1)\oplus\mcO_{Y}(D_2)$ for some divisors $D_1,D_2$ on $Y$, such a morphism $M$ can be represented as 
\begin{align}
	M&=
	\begin{pmatrix}
		a_0 & a_2 \\ a_1 & -a_0
	\end{pmatrix}
	: \mcO_Y(D_1-L)\oplus\mcO_Y(D_2-L)\to\mcO_Y(D_1)\oplus\mcO_{Y}(D_2),
	\label{eq:mor_split}
\end{align}
where $a_0$, $a_1$ and $a_2$ are global sections of $\mcO_Y(L)$, $\mcO_Y(L-D_1+D_2)$ and $\mcO_Y(L+D_1-D_2)$, respectively, satisfying $a_0^2+a_1a_2=F$. 

\begin{defin}
Let $\phi:X\to Y$ be a non-singular double cover. 
\begin{enumerate}
	\item We say that a line bundle $\mcL$ on $X$ \textit{splits} with respect to $\phi$ if $\phi_\ast\mcL$ is the direct sum of two line bundles on $Y$. 
	\item Let $\sbpic_\phi(X)$ denote the subgroup of $\pic(X)$ generated by line bundles which split with respect to $\phi$ (``s" of $\sbpic$ means ``sub" or ``split");
	\begin{align*}
		\sbpic_\phi(X):=\Big\langle [\mcL] \in \pic(X) \ \Big| \  \mbox{$\mcL$ splits with respect to $\phi$} \Big\rangle.
	\end{align*}
\end{enumerate}
\end{defin}

\begin{rem}\label{rem:generator}
If $\phi_\ast\mcL\cong\mcO_Y(D_1)\oplus\mcO_Y(D_2)$, 
then 
\begin{align*}
	\phi_\ast\big(\mcL\otimes\phi^\ast\mcO_Y(-D_2)\big)\cong\mcO_Y(D_1-D_2)\oplus\mcO_Y
\end{align*}
by projection formula. 
% Moreover we may assume either $\mcO_Y(D_1-D_2)\cong\mcO_Y$ or $\HH^0\big(Y,\mcO_Y(D_1-D_2)\big)=0$. 
Since $\phi^\ast\mcO_Y(D_2)\in\sbpic_\phi(X)$, the subgroup $\sbpic_\phi(X)$ is generated by $\phi^\ast(\Pic(Y))$ and line bundles $\mcL$ satisfying $\phi_\ast\mcL\cong\mcO_Y(D')\oplus\mcO_Y$ on $X$ 
% such that $\mcL\cong\phi^\ast\mcO_Y(D')$ or $\phi_\ast\mcL\cong\mcO_Y(D')\oplus\mcO_Y$ with $\mcO_Y(D')\cong\mcO_Y$ or $\HH^0\big(Y,\mcO_Y(D_1-D_2)\big)=0$ 
for some divisor $D'$ on $Y$. 

\end{rem}

\begin{lem}\label{lem:generator}
If $Y$ is an open subset of a smooth projective variety $\overline{Y}$ with $\codim_{\overline{Y}}(\overline{Y}\setminus Y)\geq2$, 
then $\HH^0(Y,\mcO_Y)=\CC$, and 
$\sbpic_{\phi}(X)$ is generated by $\phi^\ast(\Pic(Y))$ and line bundles $\mcL$ with $\phi_\ast\mcL\cong\mcO_Y(D')\oplus\mcO_Y$ such that either $\mcO_Y(D')\cong\mcO_Y$ or $\HH^0(Y,\mcO_Y(D'))=0$. 
\end{lem}
\begin{proof}
For an irreducible divisor $C$ on $Y$, let $\overline{C}$ denote the closure of $C$ on $\overline{Y}$. 
Let $D$ be a divisor on $Y$, and put $\overline{D}:=\sum_{i=1}^k\overline{C}_i$, where $D=\sum_{i=1}^kC_i$ is the irreducible decomposition of $D$. 
Then $\mcO_Y(D)=\mcO_{\overline{Y}}(\overline{D})|_{Y}$. 
By \cite[Proposition~1.6]{hartshorne1980}, we have $\HH^0\big(Y,\mcO_Y(D)\big)=\HH^0\big(\overline{Y},\mcO_{\overline{Y}}(\overline{D})\big)$. 
Thus 
\begin{align*}
	\HH^0(Y,\mcO_Y)=\HH^0\left(\overline{Y},\mcO_{\overline{Y}}\right)=\CC. 
\end{align*}
For a line bundle $\mcL$ on $X$ with $\phi_\ast\mcL\cong\mcO_Y(D)\oplus\mcO_Y$, if $\mcO_Y(D)\not\cong\mcO_Y$ and $\HH^0(Y,\mcO_Y(D))\ne0$, then $\HH^0(Y,\mcO_Y(-D))=0$ and 
\begin{align*}
	\phi_\ast\Big(\mcL\otimes\phi^\ast\big(\mcO_Y(-D)\big)\Big)\cong\mcO_Y(-D)\oplus\mcO_Y.
\end{align*} 
Therefore the assertion holds true by Remark~\ref{rem:generator}. 
\end{proof}

%\begin{thm}\label{thm:split}
%Let $D^+$ be an effective divisor on $X$. 
%Let $D'$ be a divisor on $Y$, and let $D$ be an effective divisor on $Y$ such that $\phi^\ast D=D^++\iota^\ast D^+$. 
%Assume the following conditions;
%\begin{enumerate}[label=\rm (\alph{enumi})]
%	\item $\HH^0(Y,\mcO_Y)=\CC$; 
%	\item either $\mcO_Y(D')\cong\mcO_Y$ or $\HH^0\big(Y,\mcO_Y(D')\big)=0$; 
%	\item $D$ is defined by $f=0$ for $f\in\HH^0\big(Y,\mcO_Y(D)\big)$. 
%\end{enumerate}
%If $\phi_\ast\mcO_X(D^+)\cong\mcO_Y(D')\oplus\mcO_Y$, 
%then $D$ satisfies the following two conditions; 
%\begin{enumerate}[label=\rm (\roman{enumi})]
%	\item $D'$ is linearly equivalent to $D-L$, i.e., $\mcO_Y(D')\cong\mcO_Y(D-L)$; and 
%	\item there are global sections $a_0$ and $a_1$ of $\mcO_Y(L)$ and $\mcO_Y(2L-D)$, respectively, such that $F=a_0^2+fa_1$. 
%\end{enumerate}
%Moreover, in the case where $D^+$ is irreducible, the converse holds true. 
%\end{thm}
\begin{proof}[Proof of Theorem~\ref{thm:split}]
Let $D'$ be a divisor on $Y$ such that either $D'\sim0$ or $\HH^0(Y,\mcO_Y(D'))=0$. 
Put $\mcL:=\mcO_X(D^+)$ and $\mcM:=\mcO_Y(D')\oplus\mcO_Y$. 
Let $\tilde{f}\in\HH^0(X,\mcL)$ be a section defining $D^+$. 
Suppose that there is an isomorphism $\upsilon:\mcM\overset{\sim}{\to}\phi_\ast\mcL$. 
By $\HH^0(Y,\mcO_Y)=\CC$ and the assumption for $D'$, we may assume that 
\begin{align*}
	\upsilon^{-1}(\tilde{f})&=\bm{b}_2:=
	\begin{pmatrix}
		0 \\ 1
	\end{pmatrix}
	\in \HH^0\big( Y,\mcM \big)=\HH^0\big(Y,\mcO_Y(D')\big)\oplus\HH^0\big(Y,\mcO_Y\big)
\end{align*}
after taking a certain basis $\bm{b}_1,\bm{b}_2$ of $\mcM$, $\mcM=\mcO_Y(D')\bm{b}_1\oplus\mcO_Y\bm{b}_2$. 
Let $M:\mcM(-L)\to\mcM$ be a morphism satisfying $M^2=F{\cdot}\id_\mcM$ such that $\mcL$ is associated to $(\mcM,M)$. 
The morphism $M$ can be regarded as a matrix with respect to the basis $\bm{b}_1,\bm{b}_{2}$ as in (\ref{eq:mor_split}) with $D_1=D'$ and $D_2=0$. 
Let $\gtU:=\{U_i\}_{i\in I}$ be an affine open covering of $Y$, and put $a_{ik}:=a_k|_{U_i}$ ($i\in I$, $k=0,1,2$) and 
\begin{align*}
	M_i&:=M|_{U_i}=
	\begin{pmatrix}
		a_{i0} & a_{i2} \\ a_{i1} & -a_{i0}
	\end{pmatrix}.
\end{align*}
Let $\eta_{ij}$ be sections of $\Gamma(U_i\cap U_j,\mcO_Y)$ which form transition functions of $\mcO_Y(D')$: 
\[
\begin{tikzpicture}
	\node (Lj) at (0,1.2) {$\mcO_Y(D')|_{U_j}$};
	\node (Li) at (2.5,1.2) {$\mcO_Y(D')|_{U_i}$};
	\node (Oj) at (0,0) {$\mcO_{U_j}$};
	\node (Oi) at (2.5,0) {$\mcO_{U_i}$};
	
	\draw[double distance=2.5pt] (Lj)--(Li);
	\draw[->] (Oj)-- node[above, scale=0.7] {$\times \eta_{ij}$} (Oi);
	\draw[->] (Lj)-- node[right] {\rotatebox{90}{$\sim$}} (Oj);
	\draw[->] (Li)-- node[right] {\rotatebox{90}{$\sim$}} (Oi);
\end{tikzpicture}
\]
Transition functions of $\mcM$ are represented by the matrices 
\begin{align*}
	G_{ij}&=
	\begin{pmatrix}
		\eta_{ij} & 0 \\ 0& 1
	\end{pmatrix}
	: \mcO_{U_j}\oplus\mcO_{U_j} \to \mcM \to \mcO_{U_i}\oplus\mcO_{U_i}. 
\end{align*}
Since $(\{G_{ij}\},\{M_i\})_{\gtU}$ is not good in general, we take a good representation $(\{G^\natu_{\alpha\beta}\},\{M^\natu_{\alpha}\})_{\gtU^\natu}$ constructed from $(\{G_{ij}\},\{M_i\})_{\gtU}$ as in the proof of Lemma~\ref{lem:good-rep}. 
Let $\varphi_\alpha:\mcM|_{U_\alpha^\natu} \to\mcO_{U_\alpha^\natu}^{\oplus2}$ ($\alpha\in I^\natu$) be isomorphisms such that $\varphi_\alpha\circ\varphi_\beta^{-1}=G_{\alpha\beta}^\natu$. 
Recall that $I^\natu:=\{(i,k)\mid i\in I,\ k=1,\dots,n_i\}$. 
Then 
\begin{align*}
	\varphi_{(i,2)}\circ\upsilon^{-1}(\tilde{f})&=
	\begin{pmatrix}
		1 \\ 0
	\end{pmatrix}, 
	&
	\varphi_{(i,k)}\circ\upsilon^{-1}(\tilde{f})&=
	\begin{pmatrix}
		0 \\ 1
	\end{pmatrix} 
	\quad (k\ne 2)
\end{align*}
for $i\in I$. 
By Corollary~\ref{cor:corr}, there exist isomorphisms $\tilde{\varphi}_{\alpha}:\mcL|_{V_\alpha}\to\mcO_{V_i}$ ($\alpha\in I^\natu$) such that
\begin{align*}
	\tilde{\varphi}_{(i,2)}(\tilde{f})&=1, 
	&
	\tilde{\varphi}_{(i,k)}(\tilde{f})&=\frac{t_i-a^\natu_{(i,k)0}}{a^\natu_{(i,k)1}} \quad (k\ne 2), 
%	&
%	\tilde{\varphi}_{(i,k)}(\tilde{f})&=\frac{t_i-a_{i0}-p_{ik}a_{i2}}{a_{i1}-2p_{ik}a_{i0}-p_{ik}^2a_{i2}}. 
\end{align*}
By Corollary~\ref{cor:splitting}, $D$ is locally defined by $a^\natu_{(i,2)1}=a_{i2}$ on $U^\natu_{(i,2)}$ and $-a^\natu_{(i,k)2}=-a_{i2}$ on $U^\natu_{(i,k)}$ for $k\ne 2$. 
Hence $a_2$ is a global section of $(\det\mcM)\otimes\mcO_Y(L)$ (note that the difference of sign of $a_{i2}$ above is derived from $\det(A_{(i,2)})=-1$ and $\det(A_{(i,k)})=1$ for $k\ne 2$ and $A_{(i,k)}$ in (\ref{eq:Palpha})). 
Since $\det\mcM\cong\mcO_Y(D')$ and $a_2\in\HH^0(Y,\mcO_Y(D))$, we obtain (i) by Proposition~\ref{prop:inverse}. 
Since $a_2$ and $f$ defines the same divisor $D$, we have $f=ca_2$ for some $c\in\CC^\times$. 
By replacing $a_1$ with $a_1/c$, we obtain $F=a_0^2+fa_1$, hence (ii) holds. 

For the last assertion, suppose that $D^+$ is irreducible, and that %(i) and 
(ii) holds. 
Put $\mcM:=\mcO_Y(D-L)\oplus\mcO_Y$. 
Then we have the morphism $M_f:\mcM(-L)\to\mcM$ represented by the matrix 
\begin{align*}
	M_f&=
	\begin{pmatrix}
		a_0 & f \\ a_1 & -a_0
	\end{pmatrix}
	:
	\mcO_Y(D-2L)\oplus\mcO_Y(-L)\to\mcO_Y(D-L)\oplus\mcO_Y 
\end{align*}
Let $\mcL_f$ be the line bundle on $X$ associated to $(\mcM,M_f)$, 
and let $\upsilon_f:\mcM\to\phi_\ast\mcL_f$ be the natural isomorphism in Corollary~\ref{cor:corr} \ref{cor:corr_morphism}. 
By the above argument, we can see that $\tilde{f}:=\upsilon_f(\bm{b}_2)$ defines a component $\widetilde{D}$ of $\phi^\ast D$. 
Since $D^+$ is irreducible, either $\widetilde{D}=D^+$ or $\widetilde{D}=\iota^\ast D^+$. 
Hence either $\mcO_X(D^+)\cong\mcL_f$ or $\mcO_X(D^+)\cong\iota^\ast\mcL_f$. 
Since $\phi_\ast\mcL_f\cong\mcM$, we obtain $\phi_\ast\mcO_X(D^+)\cong\mcM$ by Proposition~\ref{prop:inverse}. 
\end{proof}

Theorem~\ref{thm:split} implies that generators of $\sbpic_\phi(X)$ correspond to equations of the form $F=a_0^2+a_1a_2$. 
Hence we can expect that $\sbpic_\phi(X)$ reflects the arrangement of $B_\phi$ in $Y$ enough to describe the structure of $\pic(X)$. 
In several examples below, the equation $\sbpic_\phi(X)=\pic(X)$ holds (see also Example~\ref{ex:rami}). 

\begin{ex}
Let $X$ be a hyperelliptic curve, and let $\phi:X\to\PP^1$ be a non-singular double cover. 
Since any rank $2$-bundle on $\PP^1$ splits, we have $\pic(X)=\sbpic_\phi(X)$. 
\end{ex}

\begin{ex}\label{ex:conic}
Let $\phi:X\to\PP^2$ be a double cover branched along a smooth conic $B_\phi$. 
Note that $\deg(L)=1$. 
Then $X\cong\PP^1\times\PP^1$ and $\pic(X)\cong\ZZ\oplus\ZZ$. 
Let $D^+$ be a ruling of $X\cong\PP^1\times\PP^1$. 
The image $D=\phi(D^+)$ is a tangent line of $B_\phi$. 
Let $f\in\HH^0\big(\PP^2,\mcO_{\PP^2}(1)\big)$ be a section defining $D$. 
Then $B_\phi$ is given by $a_0^2+fa_1=0$ for some $a_0,a_1\in\HH^0\big(\PP^2,\mcO_{\PP^2}(1)\big)$. 
By Proposition~\ref{prop:inverse} and Theorem~\ref{thm:split}, 
we obtain 
	$\phi_\ast \mcO_X(D^+)\cong\phi_\ast\mcO_X(\iota^\ast D^+)\cong\mcO_Y^{\oplus 2}$. 
Since $\pic(X)$ is generated by $\mcO_X(D^+)$ and $\mcO_X(\iota^\ast D^+)$, we have $\pic(X)=\sbpic_{\phi}(X)$. 
By \cite{schwarzenberger1961}, it is known that $\phi_\ast\mcO_X(m D^+)$ is indecomposable if $|m|\geq 2$. 
\end{ex}

\begin{ex}\label{ex:quartic}
Let $\phi:X\to\PP^2$ be a double cover branched along a smooth quartic $B_\phi$. Note that $\deg(L)=2$. 
Then $X$ is isomorphic to the blowing-up of $\PP^2$ at $7$ points in general position. 
Moreover $\pic(X)$ is generated by $8$ $(-1)$-curves $E_0,\dots,E_7$, where $E_0$ is the strict transform of a line passing through two bowing-up centers, and $E_1,\dots,E_7$ are the exceptional divisors. 
The images $\phi(E_0),\dots,\phi(E_7)$ are $8$ of $28$ bitangent lines of $B_\phi$. 
Let $f_j\in\HH^0\big(\PP^2,\mcO_{\PP^2}(1)\big)$ be a section defining $\phi(E_j)$ for each $j=0,\dots,7$. 
Then there exist global sections $a_{j,k}$ ($k=0,1$) of $\mcO_{\PP^2}(k+2)$ such that $B_\phi$ is defined by $a_{j,0}^2+f_ja_{j,1}=0$. 
Hence $\phi_\ast\mcO_X(E_j)\cong\mcO_{\PP^2}(-1)\oplus\mcO_{\PP^2}$ for each $j=0,\dots,7$ by Theorem~\ref{thm:split}. 
Therefore we obtain $\pic(X)=\sbpic_\phi(X)$. 
In \cite{schwarzenberger1961} and \cite{valles2009}, it is shown that there are line bundles $\mcL$ on $X$ such that $\phi_\ast\mcL$ is indecomposable. 
\end{ex}

The following conjecture and problem arise. 

\begin{conj}\label{conj:spic}
If $\HH^0(Y,\mcO_Y)=\CC$, then $\pic(X)=\sbpic_\phi(X)$. 
\end{conj}

\begin{prob}
Describe the rank of $\sbpic_\phi(X)$	in terms of the branch locus $B_\phi$ and the divisor $L$. 
\end{prob}

\begin{rem}
In \cite{shirane2021}, it is proved that Conjecture~\ref{conj:spic} holds true if $Y$ is isomorphic to the projective space $\PP^n$ for any $n\geq 1$. 
\end{rem}

\section{An idea to generate $2$-bundles}\label{sec:normal}

In this section, we give an idea to generate $2$-bundles through double covers. 
We call $\phi:X\to Y$ a \textit{normal double cover} if $\phi$ is a finite surjective morphism of degree two from a normal variety $X$ to a smooth variety $Y$ over $\CC$. 
Let $\wcl(X)$ be the divisor class group of $X$. %, the set of linearly equivalent classes of Weil divisors on $X$. 
Note that %$\wcl(X)$ naturally has a group structure, and 
there is a canonical one-to-one correspondence between $\wcl(X)$ and the set of divisorial sheaves on $X$ (cf. \cite{schwede_div}). 
As the idea of \cite{catper2017}, we can apply our method to divisorial sheaves on normal double covers as follows. (See \cite{hartshorne1980} for general results on reflexive sheaves.) 

Let $\phi:X\to Y$ be a normal double cover. 
Let $X^\circ$ be the smooth locus $X\setminus\Sing(X)$ of $X$, and put $Y^\circ:=\phi(X^\circ)$. 
Then the restriction $\phi^\circ:X^\circ\to Y^\circ$ of $\phi$ is a non-singular double cover. 
For a divisorial sheaf $\mcL$ on $X$, the restriction $\mcL^\circ$ of $\mcL$ to $X^\circ$ is a line bundle on $X^\circ$, and $i_\ast\mcL^\circ=\mcL$ and $j_\ast\phi^\circ_\ast\mcL^\circ=\phi_\ast\mcL$ hold, where $i:X^\circ\to X$ and $j:Y^\circ\to Y$ are the inclusion maps. 
Hence computation of push-forwards of line bundles on $X^\circ$ can be applied to that of divisorial sheaves on $X$ via $j_\ast$. 
By modifying the proof of \cite[Theorem~3]{schwarzenberger1961}, we can prove Theorem~\ref{thm:push}. 

\begin{proof}[Proof of Theorem~\ref{thm:push}]
Let $\PP_\mcE$ be the $\PP^1$-bundle $\Proj(S(\mcE))$, and let $p:\PP_\mcE\to Y$ be the projection, where $S(\mcE)$ is the symmetric algebra of $\mcE$. 
Let $\mcH$ be a very ample line bundle on $Y$. 
The line bundle $\mcH$ gives the embedding $\Phi_{\mcH}:Y\hookrightarrow\PP^s$ with $s+1=\dim \HH^0(Y,\mcH)$. 
For $k$ large enough, we have the following exact sequence:
\begin{align*}
	\mcH^{\oplus r+1}\to\mcE\otimes\mcH^k\to 0
\end{align*}
for some $r>0$. 
This induces an embedding $i:\PP_\mcE\hookrightarrow\PP^N$ for $N=rs+r+s$ via the Segre embedding $\PP^r\times\PP^s\hookrightarrow\PP^N$. 
Put $\widetilde{\mcL}:=i^\ast\mcO_{\PP^N}(1)$. 
Note that $i(p^{-1}(P))$ is a line in $\PP^N$ for each $P\in Y$, and $p_\ast\widetilde{\mcL}\cong\mcE\otimes\mcH^k$. 
Hence $i:\PP_\mcE\hookrightarrow\PP^N$ induces an embedding $i':Y\hookrightarrow \Gr_1(N)$ by $P\mapsto i(p^{-1}(P))$, where $\Gr_1(N)$ is the Grassmannian consisting of lines in $\PP^N$. 

%This induces an embedding $\PP_\mcE\hookrightarrow\PP^r_Y\cong\PP^r\times Y$ over $Y$ which maps the fiber $p^{-1}(P)$ to a line in $\PP^r\times\{P\}$ for each $P\in P$.
%The very ample line bundle $\mcH$ gives the embedding $\Phi_{\mcH}:Y\hookrightarrow\PP^s$ with $s+1=\dim \HH^0(Y,\mcH)$. 
%We have an embedding $i:\PP_\mcE\hookrightarrow\PP^N$ via the Segre embedding $\PP^r\times\PP^s\hookrightarrow\PP^N$ for $N=rs+r+s$ such that $i(p^{-1}(P))\subset\PP^N$ is a line for each $P\in Y$. 
%Then $i:\PP_\mcE\hookrightarrow\PP^N$ induces an embedding $i':Y\hookrightarrow \Gr_1(N)$ by $P\mapsto i(p^{-1}(P))$, where $\Gr_1(N)$ is the Grassmannian consisting of lines in $\PP^N$. 
%Put $\widetilde{\mcL}:=i^\ast\mcO_{\PP^N}(1)$. %be the line bundle on $\PP_\mcE$ corresponding to a hyperplane section in $\PP^N$. 
%Then we have $p_\ast\widetilde{\mcL}\cong\mcE\otimes\mcH^{k+1}$. 

For a quadratic hypersurface $Q\subset\PP^N$, let $V(Q)$ be the subscheme of $\Gr_1(N)$ consisting of lines on $Q$. 
Note that $\PGL(N,\CC):=\Aut(\PP^N)$ acts transitively on both of $\PP^N$ and $\Gr_1(N)$ such that $V(g(Q))=g(V(Q))$ for any $g\in\PGL(N,\CC)$ and $Q\subset\PP^N$. 
Since $\dim\Gr_1(N)=2N-2$, $\dim V(Q)=2N-5$ (cf. \cite{beheshti2006}) and $\dim Y=n$, we obtain $\dim Y\cap V(Q)=n-3$ for a general hypersurface $Q\subset\PP^N$ of degree $2$ by \cite[Theorem~2]{kleimann1974}. 
Put $X':=\PP_\mcE\cap Q$ for a general quadratic hypersurface $Q\subset\PP^N$ such that $\dim i'(Y)\cap V(Q)=n-3$ and $X'$ is smooth. 
Let $\mcL'$ be the restriction of $\widetilde{\mcL}$ to $X'$. 
For an affine open set $U$ of $Y$, the K\"unneth formula for sheaves implies that 
\begin{align*}
	R^q p_\ast\widetilde{\mcL}^{-1}(U)=\HH^q(U\times \PP^1, \mcO_U\otimes\mcO_{\PP^1}(-1))=0
\end{align*}
for all $q\geq 0$. 
Since $\widetilde{\mcL}\otimes\mcJ_{X'}\cong\widetilde{\mcL}^{-1}$ for the ideal sheaf $\mcJ_{X'}\cong\widetilde{\mcL}^{-2}$ of $X'$, 
we have $p'_\ast\mcL'\cong\mcE\otimes\mcH^k$ by \cite[Proposition~5]{schwarzenberger1961}. 

Then the restriction $p':=p|_{X'}:X'\to Y$ is a generically finite morphism of degree $2$. 
Let $U':=\{P\in Y \mid \mbox{$(p')^{-1}(P)$ is finite}\}$. 
By Stein factorization of $p'$, we obtain a birational morphism $f':X'\to X''$ and a finite morphism $g':X''\to Y$ such that $g'\circ f'=p'$. 
Take the normalization $\kappa:X\to X''$, and put $\phi:=g'\circ\kappa:X\to Y$, which is a normal double cover. 
\[ 
\begin{tikzpicture}
	\node (x) at (-1.5,1.6) {$X$};
	\node (x') at (1.5,2) {$X'$};
	\node (x'') at (0,1.2) {$X''$};
	\node (y1) at (-1.5,0) {$Y$};
	\node (y2) at (0,0) {$Y$};
	\node (y3) at (1.5,0) {$Y$};
	\node (y4) at (3,0) {$Y$};
	\node (p) at (3,2) {$\PP_\mcE$};
	\node (pp) at (4.5,2) {$\PP^N$};

	\draw[->] (x') -- node[above, scale=0.7] {$\subset$} (p);
	\draw[->] (x') -- node[above, scale=0.7] {$f'$} (x'');
	\draw[->] (x) -- node[above, scale=0.7] {$\kappa$} (x'');
	\draw[->] (p) -- node[right, scale=0.7] {$p$} (y4);
	\draw[->] (x') -- node[right, scale=0.7] {$p'$} (y3);
	\draw[->] (x'') -- node[right, scale=0.7] {$g'$} (y2);
	\draw[->] (x) -- node[right, scale=0.7] {$\phi$} (y1);
	\draw[double distance=2.5pt] (y1) -- (y2);
	\draw[double distance=2.5pt] (y2) -- (y3);
	\draw[double distance=2.5pt] (y3) -- (y4);
	\draw[->] (p) --node[above, scale=0.7] {$i$} (pp);
%	
%	\node (Lj) at (0,1.2) {$\mcO_Y(-L)|_{U_j}$};
%	\node (Li) at (3,1.2) {$\mcO_Y(-L)|_{U_i}$};
%	\node (Oj) at (0,0) {$\mcO_{U_j}$};
%	\node (Oi) at (3,0) {$\mcO_{U_i}$};
%	
%	\draw[double distance=2.5pt] (Lj)--(Li);
%	\draw[->] (Oj)-- node[above, scale=0.7] {$\times \xi_{ij}$} (Oi);
%	\draw[->] (Lj)-- node[right] {\rotatebox{90}{$\sim$}} (Oj);
%	\draw[->] (Li)-- node[right] {\rotatebox{90}{$\sim$}} (Oi);
\end{tikzpicture}
\]
Let $\mcL$ be the double dual $(\kappa^\ast f_\ast\mcL')^{\vee\vee}$ of $\kappa^\ast f_\ast\mcL'$. 
Then $\mcL$ is a divisorial sheaf on $X$, and $\phi_\ast\mcL|_{U'}\cong p'_\ast\mcL'|_{U'}$ since $f'$ and $\kappa$ are isomorphic over $U'$. 
Since $\phi_\ast\mcL$ is reflexive by \cite[Corollary~1.7]{hartshorne1980} and $\codim_Y(Y\setminus U')=3$, $\phi_\ast\mcL\cong p'_\ast\mcL'\cong\mcE\otimes\mcH^k$. 
Therefore, $\phi_\ast(\mcL\otimes\phi^\ast\mcH^{-k})\cong\mcE$. 
\end{proof}

{\rev
Let $\phi:X\to Y$ be a normal double cover over a smooth projective variety $Y$ branched at $B\subset Y$, and let $F$ be a global section of $\mcO_Y(B)$ defining $B$. 
If $F$ has several representations of the form $F=a_0^2+a_1a_2$, then we can expect that many $2$-bundles on $Y$ are generated by the following method: 
\begin{enumerate}[label=\rm (\roman{enumi})]
	\item Let $\phi^\circ:X^\circ\to Y^\circ$ be the non-singular double cover as above;
	\item take several line bundles $\mcL_1,\dots,\mcL_m$ on $X^\circ$ such that $\phi^\circ_\ast\mcL_i$ is split, and compute $2$-bundles $\phi^\circ_\ast(\mcL_1^{n_1}\otimes\dots\otimes\mcL_m^{n_m})$ on $Y^\circ$ by Theorem~\ref{thm:group_law};
	\item\label{method:fainal} then $j_\ast\phi^\circ_\ast(\mcL_1^{n_1}\otimes\dots\otimes\mcL_m^{n_m})$ are reflexive sheaves of rank two. 
\end{enumerate}
If $\phi:X\to Y$ is non-singular, then the reflexive sheaves in \ref{method:fainal} are $2$-bundles. 
This methods has the following problem:

\begin{prob}
	\begin{enumerate}[label=\rm (\arabic{enumi})]
		\item Give a condition for a reflexive sheaf $j_\ast\phi^\circ_\ast(\mcL_1^{n_1}\otimes\dots\otimes\mcL_m^{n_m})$ in \ref{method:fainal} to be a $2$-bundle. 
		\item Which normal double cover $\phi:X\to Y$ generates many $2$-bundles on $Y$ by the above method? 
		%Give a method for constructing reduced divisors $B$ defining $F=0$ with several representations of the form $F=a_0^2+a_1a_2$. 
	\end{enumerate}
\end{prob}
If Conjecture~\ref{conj:spic} is true, then there is a normal double cover $\phi_\mcE :X\to Y$ for any $2$-bundle $\mcE$ on a smooth projective variety $Y$ such that $\mcE$ is generated by the above method using $\phi_\mcE$. 
}

%If Conjecture~\ref{conj:spic} is true, then Theorem~\ref{thm:push} implies that any $2$-bundle on $\PP^n$ can be generated by the following method:
%\begin{enumerate}[label=\rm (\roman{enumi})]
%	\item Take a reduced divisor $B:F=0$ of even degree on $\PP^n$ with several {\rev representations} of the form $F=a_0^2+a_1a_2$;
%	\item let $\phi:X\to\PP^n$ be the normal double cover branched at $B$, and let $\phi^\circ:X^\circ\to Y^\circ$ be the non-singular double cover as above;
%	\item take several line bundles $\mcL_1,\dots,\mcL_m$ on $X^\circ$ such that $\phi^\circ_\ast\mcL_i$ is split, and compute $2$-bundles $\phi^\circ_\ast(\mcL_1^{n_1}\otimes\dots\otimes\mcL_m^{n_m})$ on $Y^\circ$ by Theorem~\ref{thm:group_law};
%	\item\label{method:fainal} then $j_\ast\phi^\circ_\ast(\mcL_1^{n_1}\otimes\dots\otimes\mcL_m^{n_m})$ are reflexive sheaves of rank two. 
%\end{enumerate}
%If $\phi:X\to\PP^n$ is non-singular, then the reflexive sheaves in \ref{method:fainal} are $2$-bundles. 
%A problem of this method is when a reflexive sheaf of \ref{method:fainal} is a $2$-bundle. 
%\begin{prob}
%	Give a condition for a reflexive sheaf $j_\ast\phi^\circ_\ast(\mcL_1^{n_1}\otimes\dots\otimes\mcL_m^{n_m})$ in \ref{method:fainal} to be a $2$-bundle. 
%\end{prob}

\section{Direct summands of $2$-bundles on $\PP^1$}\label{sec:P1}

It is known due to Grothendieck that any $2$-bundle on $\PP^1$ splits. 
Let $\Gr_1(n)$ be the Grassmannian consisting of lines in $\PP^n$. 
We identify a point of $\Gr_1(n)$ with the line in $\PP^n$ corresponding to the point. 
For a $2$-bundle $\mcE$ on $\PP^n$, we have a map $a_{\mcE}:\Gr_1(n)\to\ZZ^2$ given by $a_\mcE(L)=(k_1,k_2)$ for $L\in\Gr_1(n)$ if 
\[ \mcE|_L\cong\mcO_{\PP^1}(k_1)\oplus\mcO_{\PP^1}(k_2) \quad (k_1\geq k_2). \]
It is known that a maximal subset $U_\mcE$ of $\Gr_1(n)$ with the restriction $a_\mcE|_{U_\mcE}:U_\mcE\to\ZZ^2$ constant is a non-empty Zariski-open subset of $\Gr_1(n)$ (cf. \cite{okonek2011}). 
%It is known that $U_\mcE\subset\Gr_1(n)$ is non-empty (cf. \cite{okonek2011}). 
A line $L\subset\PP^n$ is called a \textit{jumping line} of $\mcE$ if $L\not\in U_\mcE$. 
It is clear that $\mcE$ is indecomposable if $ U_\mcE\subsetneq\Gr_1(n)$. 

In this section, we investigate a method of computing direct summands of a $2$-bundle from transition functions $G_{ij}$. 
This method enables us to compute jumping lines of a $2$-bundle on $\PP^n$ if we know transition functions of the $2$-bundle. 
For this aim, we give a proof of the well-known result of Grothendieck by using transition functions. 

We first construct a matrix $\Eu_x(G)$ for $G\in\GL(2,\CC(x))$ such that $\Eu_x(G)G$ is an upper triangular matrix. 
For $p\in\CC$, let $v_p:\CC(x)\setminus\{0\}\to\ZZ$ be the valuation at $p\in\Aa^1=\CC$, and put $v_p(0):=\infty$. 
For $G:=(g_{ij})\in\GL(2,\CC)$, let 
\[ \gd_x(G):=\prod_{p\in\CC}(x-p)^{\min\{v_p(g_{ij})\mid i,j=1,2\}}. \]
Let $G,J\in\GL(2,\CC(x))$ be matrices as follows;
\begin{align}\label{eq:def_G}
	G&:=\gd_x(G)
	\begin{pmatrix}
		a & b \\ c & d
	\end{pmatrix},
	\qquad
	J:=
	\begin{pmatrix}
		0&1\\1&0
	\end{pmatrix}.
\end{align}
Note that we have $\gd_x(G)\ne0$ and $a,b,c,d\in\CC[x]$ with $\gcd(a,b,c,d)=1$. 
Put $A_k$ for $k\geq 0$ as~$A_0:=E$,
\begin{align*}
	G_k&:=A_{k}G=\gd_x(G) 
	\begin{pmatrix}
		a_k & b_k \\ c_k & d_k
	\end{pmatrix},
	\\[0.5em]
	A_{k+1}&:=
	\left\{
	\begin{array}{ll}
%		\begin{pmatrix}
%			0 & 1 \\ 1 & 0
%		\end{pmatrix}
		JA_{k}
		& \mbox{if $\deg_x(a_{k})>\deg_x(c_{k})$ \ or \ $a_k=0$,}
		\\[1.0em]
		\begin{pmatrix}
			1 & 0 \\ s_{k+1} & 1
		\end{pmatrix}A_{k}
		& \mbox{if $a_k,c_k\ne0$ and $\deg_x(a_{k})\leq\deg_x(c_{k})$},
		\\[1em]
		A_k
		& \mbox{if $c_k=0$,}
	\end{array}
	\right.
\end{align*}
where $s_k\in\CC[x]$ is the polynomial with $\deg_x(c_{k-1}+s_ka_{k-1})<\deg_x(a_{k-1})$. 
By Euclidian algorithm, $c_m=0$ for some $m\geq 0$. 
Note that $a_md_m=\pm\det(G)\ne0$. 
Let $a_{m,0},d_{m,0}\in\CC^\times$ be the coefficients of the leading terms of $a_m, d_m$, respectively. 
Put 
\begin{align*}
	B_m&:=
	\begin{pmatrix}
		a_{m,0}^{-1} & s_m \\ 0 & d_{m,0}^{-1}
	\end{pmatrix}, 
	&
	\Eu_x(G)&:=B_mA_m,
\end{align*}
where $s_m$ is the polynomial satisfying $\deg_x(a_{m,0}^{-1}b_m+s_md_m)<\deg_x(d_m)$. 
Then $\Eu_x(G)G$ is of the form 
\begin{align}
	\Eu_x(G)G&=\gd_x(G)
	\begin{pmatrix}
		a' & b' \\ 0 & d'
	\end{pmatrix}
	\label{eq:euclid_1}
\end{align}
such that $a',d'\in\CC[x]$ are monic, $\deg_x(b')<\deg_x(d')$ and $a'=\gcd(a,c)$. 
For an affine open $U\subset\Aa^1$ and $a\in\CC(x)\setminus\{0\}$, let 
\[ \den_U(a):= \prod_{p\in\Aa^1\setminus U}(x-p)^{v_p(a)}, \qquad \neu_U(a):=\frac{a}{\den_U(a)}. \]
Note that $\den_U(a)\in\Gamma(U,\mcO_U^\times)$, and $\den_U(a)\den_U(b)=\den_U(ab)$ and $\neu_U(a)\neu_U(b)=\neu_U(ab)$ for $a,b\in\CC(x)\setminus\{0\}$. 

Let $\gtU:=\{U_i\}_{i=0}^n$ be an affine open covering of $\Aa^1$, and let $G_{ij}\in\GL(2,\mcO_{U_i\cap U_j})$ ($i,j=0,\dots,n$) be transition functions of a $2$-bundle $\mcE$ on $\Aa^1$. 
We construct matrices $A_i\in\GL(2,\mcO_{U_i})$ such that $A_i^{-1}G_{ij}A_j=E$. 
We regard $\CC[x]$ as the coordinate ring of $\Aa^1$. 
Put $A_0^{(0)}:=E$. 
We define $A_i^{(k)}$ for $0\leq i\leq k\leq n$ in Algorithm~\ref{algo1} below. 
Let us write $G_{0k}^{(k-1)}:=\Eu_x(G_{k0}A_0^{(k-1)})G_{k0}A_0^{(k-1)}$ as 
\[ G_{k0}^{(k-1)}=g_k
\begin{pmatrix}
	a_k & b_k \\ 0 & d_k
\end{pmatrix},
\qquad \mbox{where \ } g_k:=\gd_x(G_{k0}^{(k-1)}). 
 \]
Since $\gcd(\neu_{U_k}(a_k), \den_{U_k}(d_k))=1$, there are polynomials $\alpha_k,\beta_k\in\CC[x]$ such that $\neu_{U_k}(a_k)\den_{U_k}(d_k)b_k+\alpha_k\neu_{U_k}(a_k)+\beta_k\den_{U_k}(d_k)=0$. 
\begin{algorithm}[H]
\caption{Computation of a global basis for a $2$-bundle on $\Aa^1$}	
\label{algo1}
\begin{algorithmic}[1]
\REQUIRE The transition functions $G_{ij}$ ($i,j=0,\dots,n$) of a $2$-bundle on $\Aa^1$.
\ENSURE The matrices $A_i$ ($i=0,\dots,n$) such that $A_i^{-1}G_{ij}A_j=E$.
\STATE $A_0^{(0)}:=E$; $k:=1$;
\WHILE{$k\leq n$}
	\STATE $i:=0$;
	\WHILE{$i\leq k-1$}
	\STATE $A_i^{(k)}:=\frac{1}{\neu_{U_k}(g_k a_k d_k)}A_i^{(k-1)}
		\begin{pmatrix}\neu_{U_k}(d_k) & \alpha_k \\ 0 & \neu_{U_k}(a_k)\end{pmatrix}$;
		\STATE $i:=i+1$;
	\ENDWHILE
	\STATE $A_k^{(k)}:=\left( \frac{1}{\den_{U_k}(g_k a_k d_k)}
	\begin{pmatrix}
	\den_{U_k}(d_k) & \beta_k \\ 0 & \den_{U_k}(a_k)
	\end{pmatrix}
	\Eu_x(G_{k0}A_0^{(k-1)})
	\right)^{-1}$;
	\STATE $k:=k+1$;
\ENDWHILE
\FOR{$i=0,\dots,n$}
\STATE $A_i:=A_i^{(n)}$;
\ENDFOR
\RETURN $A_0,\dots,A_n$.
\end{algorithmic}
\end{algorithm}

%Put 
%\begin{align*}
%A_i^{(k)}&:=\frac{1}{\neu_{U_k}(g_k a_k d_k)}A_i^{(k-1)}
%\begin{pmatrix}
%	\neu_{U_k}(d_k) & \alpha_k \\ 0 & \neu_{U_k}(a_k)
%\end{pmatrix}
%\quad (0\leq i\leq k-1),
%\\
%A_k^{(k)}&:=\left( \frac{1}{\den_{U_k}(g_k a_k d_k)}
%\begin{pmatrix}
%	\den_{U_k}(d_k) & \beta_k \\ 0 & \den_{U_k}(a_k)
%\end{pmatrix}
%\Eu_x(G_{k0}A_0^{(k-1)})
% \right)^{-1}.
%\end{align*}
%Put $A_i:=A_i^{(n)}$. Then the following proposition holds. 

\begin{prop}\label{prop:A1} 
For $A_i$ $(i=0,\dots,n)$ obtained by Algorithm~\ref{algo1}, the followings hold. 
\begin{enumerate}[label=\rm (\roman{enumi})]
	\item $A_i\in\GL(2,\mcO_{U_i})$ for $i=0,\dots,n$. 
	\item $A_i^{-1}G_{ij}A_j=E$ for $i,j=0,\dots,n$. 
\end{enumerate}	
In particular, any $2$-bundle on $\Aa^1$ is isomorphic to $\mcO_{\Aa^1}^{\oplus 2}$. 
\end{prop}
\begin{proof}
Note that $A_i^{(k)}\in\GL(2,\CC(x))$. 
We first prove $G_{i0}A_0^{(k)}=A_i^{(k)}$ for $0\leq i\leq k\leq n$. 
In the case of $k=0$, the equation is trivial. 
For $k\geq 1$, suppose $G_{i0}A_0^{(k-1)}=A_i^{(k-1)}$ for $i=0,\dots,k-1$.  
By the definition of $A_i^{(k)}$, we have $(A_i^{(k-1)})^{-1}A_i^{(k)}=(A_0^{(k-1)})^{-1}A_0^{(k)}$ for any $i=0,\dots,k-1$. 
Thus we obtain 
\begin{align*}
	G_{i0}A_0^{(k)}=G_{i0}A_0^{(k-1)}(A_i^{(k-1)})^{-1}A_i^{(k)}=A_i^{(k)}
\end{align*}
for $i=0,\dots,k-1$. 
By direct computation, we can see $(A_k^{(k)})^{-1}G_{k0}A_0^{(k)}=E$. 
Hence $G_{i0}A_0^{(k)}=A_i^{(k)}$ holds for each $i=0,\dots,k$. 
In particular, $A_i^{-1}G_{i0}A_0=E$ for $i=0,\dots,n$. 
Thus 
\[ A_i^{-1}G_{ij}A_j=A_i^{-1}G_{i0}A_0(A_j^{-1}G_{j0}A_0)^{-1}=E, \]
and (ii) holds. 

We next prove (i) by induction on $k$. 
Assume that $A_i^{(k-1)}\in\GL(2,\mcO_{U_i})$ for $0\leq i\leq k-1$. 
By definitions of $\Eu_x(G)$ and $\den_{U}(a)$, we have $A_k^{(k)}\in\GL(2,\mcO_{U_k})$. 
By the above argument, we obtain 
\[ \GL(2,\mcO_{U_0\cap U_k})\ni G_{k0}A_0^{(k-1)} =G_{ki}G_{i0}A_{0}^{(k-1)}=G_{ki}A_i^{(k-1)}\in\GL(2,\mcO_{U_i\cap U_k}). \]
This implies that $G_{k0}A_0^{(k-1)}\in\GL(2,\mcO_{U^{(k)}\cap U_k})$, where $U^{(k)}:=U_0\cup\dots\cup U_{k-1}$. 
Hnece $G_{k0}^{(k-1)}\in\GL(2,\mcO_{U^{(k)}\cap U_k})$, and $g_k,a_k,d_k\in\Gamma(U^{(k)},\mcO_{U^{(k)}\cap U_k}^\times)$. 
Therefore $\neu_{U_k}(g_k), \neu_{U_k}(a_k),\neu_{U_k}(d_k)\in\Gamma(U^{(k)},\mcO_{U^{(k)}}^\times)$, and $A_i^{(k)}\in\GL(2,\mcO_{U_i})$ for $i=0,\dots,k-1$. 
\end{proof}

Next we compute the direct summands of a $2$-bundle $\mcE$ on $\PP^1$ from its transition functions. 
Let $u_0,u_1$ be homogeneous coordinates of $\PP^1$, and let $U_i\subset\PP^1$ be the affine open subset defined by $u_i\ne 0$ for $i=0,1$. 
Put 
\[ x:=\frac{u_1}{u_0}, \qquad y:=\frac{u_0}{u_1}. \]
Note that $\CC(\PP^1)=\CC(x)=\CC(y)$. 
Let $P_x,P_y\in\PP^1$ be the points defined by $u_1=0$ and $u_0=0$, respectively. 
By Proposition~\ref{prop:A1}, we may assume that a $2$-bundle on $\PP^1$ is defined by a transition function $G\in\GL(2,\mcO_{U_0\cap U_1})$. 

Let $G\in\GL(2,\mcO_{U_0\cap U_1})$ be as in (\ref{eq:def_G}). 
For $P\in\PP^1$, let $v_P:\CC(x)\setminus\{0\}\to\ZZ$ be the valuation of $P\in\PP^1$, and put $v_P(0):=\infty$, and  
let $v_{1}(G), v_{2}(G), v(G)\in\ZZ$ be the integers
\begin{align*}
	v_{1}(G)&:=\min\big\{v_{P_x}(a), v_{P_x}(c)\big\}, \\
	v_{2}(G)&:=\min\big\{v_{P_x}(b), v_{P_x}(d) \big\}, \\
	v(G)&:=\min\big\{v_{P_x}(a),v_{P_x}(b),v_{P_x}(c),v_{P_x}(d)\big\}. 
%	&
%	\mbox{where } 
%	G=
%	\begin{pmatrix}
%		a & b \\ c & d
%	\end{pmatrix}.
\end{align*}
We define $e_{1}(G), e_{2}(G)\in\ZZ$ as
\begin{align*}
	e_{1}(G):=v_{P_x}(a'), \qquad e_{2}(G):=v_{P_x}(d'),
	%\left\{\!\!
	%\begin{array}{ll}
	%	&
	%	\mbox{if $
	%	\min\{v_{P}(a), v_{P}(c)\}\geq\min\{v_{P}(b), v_{P}(d)\}
	%	$,}
	%	\\[0.5em]
	%	\big( v_P(b''), \ v_P(c'') \big)
	%	& \mbox{otherwise},
	%\end{array}
	%\right.
\end{align*}
where $a',d'$ are in (\ref{eq:euclid_1}). 
Note that $v_{1}(G)=e_{1}(G)$ and $v_{2}(G)\leq e_{2}(G)$. 
%, and $b'',c''$ are given by
%\begin{align*}
%	\Eu_x(GJ%\begin{pmatrix} 0&1\\1&0 \end{pmatrix}
%	)
%	GJ%\begin{pmatrix} 0&1\\1&0 \end{pmatrix}
%	&=
%	\gd_x(G)
%	\begin{pmatrix}
%		b'' & a'' \\ 0 & c ''
%	\end{pmatrix}.
%\end{align*}
Since $G\in\GL(2,\mcO_{U_0\cap U_1})$, we have $\det(G)=x^r$ for some $r\in\ZZ$ and $\gd_x(G)=x^{v(G)}$. %, where $e_x:=v_{P_x}(G)$. 
Note that $2v(G)\leq r=v_{P_x}(\det(G))$. 
Let $b(G)\in\CC[x]$ be the polynomial such that $x^{v(G)}b(G)$ is the $(1,2)$ entry of the triangular matrix $\Eu_x(G)G$. 
Note that, if $G\in\GL(2,\mcO_{U_0\cap U_1})$, then the $(j,j)$ entry of $\Eu_x(G)G$ is $x^{e_j(G)}$ for each $j=1,2$, and $y^{e_2(G)}b(G)\in\CC[y]$ by the definition of $\Eu_x(G)$. 
For $h,p\in\CC[y]$ ($p\ne 0$), let $\quot_y(h,p)$ be the quotient of the division of $h$ by $p$ in $\CC[y]$, i.e., $\deg_y(h-p\quot_y(h,p))<\deg p$. 
By Algorithm~\ref{algo2} below, we can compute matrices $A_x\in\GL(2,\CC[x])$ and $A_y\in\GL(2,\CC[y])$ such that $A_x^{-1}GA_y$ is a diagonal matrix. 

%
%Put $A_x^{(0)}=A_y^{(0)}:=E$ and $G^{(0)}:=G$. 
%We define $A_x^{(i+1)}$ and $A_y^{(i+1)}$ by Algorithm~\ref{algo2} below. 
%
%Put $G^{(i)}:=(A_x^{(i)})^{-1}GA_y^{(i)}$. 
%Let $e^{(i)}_1,e_2^{(i)}$ be the first and second entries of $\VE_{P_x}(G^{(i)})$, respectively, %$\VE_{P_x}(G^{(i)})=(e_1^{(i)},e_2^{(i)})$, 
%and let $b^{(i)}\in\CC[x]$ be the polynomial such that $\gd_x(G^{(i)})b^{(i)}$ is the $(1,2)$ entry of the triangular matrix $\Eu_x(G^{(i)})G^{(i)}$. 

%\begin{figure*}[h]
\begin{algorithm}[H]
\caption{Computation of direct summands of a $2$-bundle on $\PP^1$}
\label{algo2}
\begin{algorithmic}[1]
\REQUIRE A transition function $G$ of a $2$-bundle $\mcE$ on $\PP^1$. 
\ENSURE Matrices $A_x,A_y$ such that $A_x^{-1}GA_y$ is diagonal, and $e_1,e_2\in\ZZ$ such that $\mcE\cong\mcO_{\PP^1}(e_1)\oplus\mcO_{\PP^1}(e_2)$. 
\STATE $G':=G$; \ $A_x:=E$; \ $A_y:=E$; 
\WHILE{$v_{1}(G')< v_{2}(G')$ or $e_{1}(G')< e_{2}(G')$}
	\IF{$v_{1}(G')< v_{2}(G')$}
	\STATE $A_x:=A_x$; \ $A_y:=A_yJ$; \ $G':=A_x^{-1}GA_y$; 
	\ELSE
	\STATE $s_y:=\quot_y\left(y^{e_{2}(G')}b(G'),\ y^{e_{2}(G')-e_{1}(G')}\right)$; 
	\STATE $A_x:=A_x\Eu_x(G')^{-1}$; \ $A_y:=A_y\begin{pmatrix} 1 & -s_y \\ 0 & 1 \end{pmatrix}$; \ $G':=A_x^{-1}GA_y$; 
	\ENDIF
\ENDWHILE
\STATE $A_x:=A_x\Eu_x(G')^{-1}$; \ $A_y:=A_y\begin{pmatrix} 1 & -y^{e_{1}(G')}b(G') \\ 0 & 1 \end{pmatrix}$; \ $G':=A_x^{-1}GA_y$; 
\STATE $e_1:=v(G')+e_{1}(G')$; $e_2:=v(G')+e_{2}(G')$; 
\RETURN $A_x,A_y$, $e_1,e_2$. 
\end{algorithmic}
\end{algorithm}
%\end{figure*}

%\begin{enumerate}[label={\rm (A\arabic{enumi})}]
%	\item\label{alg1} If $v_{P_x,1}(G^{(i)})<v_{P_x,2}(G^{(i)})$, then put 
%	\[ A_x^{(i+1)}:=A_x^{(i)}, \qquad A_y^{(i+1)}:=A_y^{(i)}J.  %G^{(i+1)}:=(A_x^{(i+1)})^{-1}G^{(i)}A_y^{(i+1)}. 
%	\]
%	\item\label{alg2} If $v_{P_x,1}(G^{(i)})\geq v_{P_x,2}(G^{(i)})$ and $e_1^{(i)}<e_2^{(i)}$, then put 
%	\[ A_x^{(i+1)}:=A_x^{(i)}\Eu_x(G^{(i)})^{-1}, \qquad  A_y^{(i+1)}:=A_y^{(i)}\begin{pmatrix} 1 & s_y^{(i)} \\ 0 & 1 \end{pmatrix}, %\  H^{(i)}:=(B_x^{(i)})^{-1}G^{(i)}B_y^{(i)}, 
%	\]
%	where $s_y^{(i)}\in\CC[y]$ is the polynomial satisfying 
%	\[ \deg_y\left(y^{e_2^{(i)}}b^{(i)}+y^{e_2^{(i)}-e_1^{(i)}}s_y^{(i)}\right)<e_2^{(i)}-e_1^{(i)}. \]
%%	and put
%%	\[ A_x^{(i+1)}:=B_x^{(i)}\Eu_x(H^{(i)}J)^{-1}, \quad A_y^{(i+1)}:=B_y^{(i)}J. %\begin{pmatrix} 1 & -y^{e_1^{(i)}}b^{(i)} \\ 0 & 1 \end{pmatrix}. 
%%	\]
%	\item\label{alg3} If $v_{P_x,1}(G^{(i)})\geq v_{P_x,2}(G^{(i)})$ and $e_1^{(i)}\geq e_2^{(i)}$, then put 
%	\[ A_x:=A_x^{(i)}\Eu_x(G^{(i)})^{-1}, \quad A_y:=A_y^{(i)}\begin{pmatrix} 1 & -y^{e_1^{(i)}}b^{(i)} \\ 0 & 1 \end{pmatrix}, \]
%	and this algorithm ends. 
%\end{enumerate}

\begin{prop}\label{prop:P1_summand}
Let $\mcE$ be a $2$-bundle on $\PP^1$ with a transition function $G\in\GL(2,\mcO_{U_0\cap U_1})$, and let $A_x,A_y,e_1,e_2$ be the result of Algorithm~\ref{algo2} for $G$. 
\begin{enumerate}[label={\rm (\roman{enumi})}]
	\item Algorithm~\ref{algo2} halts.  
	\item $A_x\in\GL(2,\CC[x])$ and $A_y\in\GL(2,\CC[y])$;
	\item $A_x^{-1}GA_y=\begin{pmatrix} x^{e_1} & 0 \\ 0 & x^{e_2} \end{pmatrix}$. %is a diagonal matrix in $\GL(2,\mcO_{U_x\cap U_y})$, and the $2$-bundle $\mcE$ on $\PP^1$ is isomorphic to $\mcO_{\PP^1}(e_1)\oplus\mcO_{\PP^1}(e_2)$.
\end{enumerate}
In particular, $\mcE\cong\mcO_{\PP^1}(e_1)\oplus\mcO_{\PP^1}(e_2)$. 
\end{prop}
\begin{proof}
Put $A_x^{(0)}=A_y^{(0)}:=E$ and $G^{(0)}:=G$. 
Let $A_x^{(i)}, A_y^{(i)}$ and $G^{(i)}$ be the matrices $A_x,A_y$ and $G'$ in the $i$th step of Algorithm~\ref{algo2}. 
Put $e_j^{(i)}:=e_j(G^{(i)})$ for $j=1,2$. 
It is clear that $A_x^{(i)}\in\GL(2,\CC[x])$ and $A_y^{(i)}\in\GL(2,\CC[y])$ for each $i$. 
Hence (ii) holds true. 
If  $v_{1}(G^{(i)})\geq v_{2}(G^{(i)})$ and $e_1^{(i)}\geq e_2^{(i)}$, then we have 
\begin{align*}
	G^{(i+1)}%=(A_x^{(i)})^{-1}GA_y^{(i)}
	&=
	x^{v(G^{(i)})}\begin{pmatrix} x^{e_1^{(i)}} & b(G^{(i)}) \\ 0 & x^{e_2^{(i)}} \end{pmatrix}\begin{pmatrix} 1 & -y^{e_1^{(i)}}b(G^{(i)}) \\ 0 & 1 \end{pmatrix}
	\\&=
	x^{v(G^{(i)})}\begin{pmatrix} x^{e_1^{(i)}} & 0 \\ 0 & x^{e_2^{(i)}} \end{pmatrix}. 
\end{align*}
On the other hand, $e_1:=v(G^{(i)})+e_1^{(i)}$ and $e_2:=v(G^{(i)})+e_2^{(i)}$. 
Thus (iii) holds true if (i) is also true. 

Suppose that $v_{1}(G^{(i)})\geq v_{2}(G^{(i)})$ and $e_1^{(i)}<e_2^{(i)}$. 
%then $v_{P_x,1}(G^{(i+1)})<v_{P_x,2}(G^{(i+1)})$ and $e_x^{(i+1)}>e_x^{(i)}$. 
Then we obtain
\begin{align*} 
G^{(i+1)}%=x^{e_x^{(i)}}\begin{pmatrix} x^{e_1^{(i)}} & b^{(i)} \\ 0 & x^{e_2^{(i)}} \end{pmatrix}\begin{pmatrix} 1 & s_y^{(i)} \\ 0 & 1 \end{pmatrix}
&=
x^{v(G^{(i)})}\begin{pmatrix} x^{e_1^{(i)}} & b(G^{(i)})-x^{e_1^{(i)}}s_y \\ 0 & x^{e_2^{(i)}} \end{pmatrix}
\\&=
x^{v(G^{(i)})+e_1^{(i)}}\begin{pmatrix} 1 & x^{e_2^{(i)}-e_1^{(i)}}\left(y^{e_2^{(i)}}b(G^{(i)})-y^{e_2^{(i)}-e_1^{(i)}}s_y\right) \\ 0 & x^{e_2^{(i)}-e_1^{(i)}} \end{pmatrix}.
 \end{align*}
By the definition of $s_y$, we have $x^{e_2^{(i)}-e_1^{(i)}}(y^{e_2^{(i)}-e_1^{(i)}}s_y+y^{e_2^{(i)}}b(G^{(i)}))\in\CC[x]$. 
Hence $v(G^{(i+1)})=v(G^{(i)})+e_1^{(i)}$ and $0=v_{1}(G^{(i+1)})<v_{2}(G^{(i+1)})$. 
Then we obtain $v(G^{(i+2)})=v(G^{(i)})+e_1^{(i)}$ and $e_1^{(i+2)}=v_{1}(G^{(i+2)})>v_{2}(G^{(i+2)})= 0$. 
If $e_1^{(i+2)}<e_2^{(i+2)}$, then $v(G^{(i+3)})=v(G^{(i)})+e_1^{(i)}+e_1^{(i+2)}>v(G^{(i)})$ by the above argument. 
Thus if the while loop in Algorithm~\ref{algo2} does not halt, then we obtain a sequence of integers 
\[ v(G^{(i)})<v(G^{(i+3)})<\dots<v(G^{(i+3n)})<\cdots, \]
which is a contradiction to $2v(G^{(i)})\leq v_{P_x}(\det(G))$ for any $i$. 
Therefore (i) holds.
\end{proof}

\section{Example}\label{sec:example}

In this section, we prove the result \cite[Proposition~8]{schwarzenberger1961} about jumping lines of the push-forward of a line bundle on a non-singular double cover by using our method. %, i.e., we compute direct summands of $\phi_\ast\mcL|_L$ by using our method for a double cover $\phi:X\to\PP^2$ branched along a smooth conic, line bundles $\mcL$ on $X$ and lines $L\subset\PP^2$. 
Our method is somewhat complicated. However it may be extended for non-singular double covers with branch curves of higher degree. 
In the last of this section, we compute global sections of a line bundle on the non-singular double cover $X$. 
Lemmas~\ref{lem:rest_push}, \ref{lem:str_sh} and \ref{lem:p1} below are useful for our aim.  

\begin{lem}\label{lem:rest_push}
Let $\phi:X\to Y$ be a flat cover with $X$ and $Y$ smooth, and let $Y'\subset Y$ be subvariety with the inclusion $i:Y'\to Y$. 
Let $X'$ be the fiber product $X\times_YY'$, and let $\phi':X'\to Y'$ and $j:X'\to X$ be the projections. 
Then $i^\ast(\phi_\ast\mcL)\cong\phi'_\ast(j^\ast\mcL)$ for a line bundle $\mcL$ on $X$. 
\end{lem}
\begin{proof}
%Since there is a natural map $i^\ast(\phi_\ast\mcL)\to\phi'_\ast(j^\ast\mcL)$, 
The question is local on $Y$ and $Y'$. 
Since $\phi$ is finite, we may assume that $X=\mathrm{Spec}\,\gtB$, $Y=\mathrm{Spec}\,\gtA$ and $Y'=\mathrm{Spec}\,\gtA/\gtI$, where $\gtA,\gtB$ are rings and $\gtI\subset \gtA$ is an ideal.  
Then $\mcL$ is the sheafification $\gtB^\sim$ of $\gtB$. 
Since $\phi$ is flat, we have the exact sequence
\begin{align*}
	0\to \gtB\otimes_\gtA \gtI \to \gtB \to \gtB\otimes_\gtA \gtA/\gtI \to 0.
\end{align*}
Hence we have $i^\ast(\phi_\ast\mcL)=(\gtB\otimes_\gtA \gtA/\gtI)^\sim\cong(\gtB/\gtI\gtB)^\sim=\phi'_\ast(j^\ast\mcL)$. 
\end{proof}

\begin{lem}\label{lem:str_sh}
	Let $\phi:X\to\PP^n$ be a non-singular double cover branched at a smooth hypersurface of degree $2l$ on $\PP^n$, and let $\mcL$ be a line bundle on $X$. 
	Then $\phi_\ast\mcL\cong\mcO_{\PP^n}\oplus\mcO_{\PP^n}(-l)$ if and only if $\mcL\cong\mcO_X$. 
\end{lem}
\begin{proof}
Put $\mcM:=\mcO_{\PP^n}\oplus\mcO_{\PP^n}(-l)$. 
Suppose that $\phi_\ast\mcL\cong\mcM$. 
Let $M:\mcM(-l)\to\mcM$ be a morphism such that $(\mcM,M)$ is an admissible pair for $\phi$ and $\mcL\cong\mcL_{(\mcM,M)}$. 
Then $M$ can be represented by a matrix 
\begin{align*}
	M=
	\begin{pmatrix}
		a_0 & a_2 \\ a_1 & -a_0
	\end{pmatrix}:\mcO_{\PP^n}(-l)\oplus\mcO_{\PP^n}(-2l)\to\mcO_{\PP^n}\oplus\mcO_{\PP^n}(-l),
\end{align*}
where $a_0$, $a_1$ and $a_2$ are global sections of $\mcO_{\PP^n}(l)$, $\mcO_{\PP^n}$ and $\mcO_{\PP^n}(2l)$, respectively. 
In particular, $F:=a_0^2+a_1a_2$ defines the branch locus, and $a_1\in\CC^\times$. 
Let $A:\mcM\to\mcM$ be the isomorphism represented as
\begin{align*}
	A=
	\begin{pmatrix}
		1 & a_0 \\ 0 & a_1
	\end{pmatrix},
	\qquad \mbox{then} \qquad 
	A^{-1}MA=
	\begin{pmatrix}
		0 & F \\ 1 & 0
	\end{pmatrix}. 
\end{align*}
Hence $\mcL\cong\mcO_X$ by Proposition~\ref{prop:corr}. 
It is known that $\phi_\ast\mcO_X\cong\mcO_{\PP^n}\oplus\mcO_{\PP^n}(-l)$, and the assertion has proved. 
\end{proof}

\begin{lem}\label{lem:p1}
Let $\phi:X\to\PP^1$ be a non-singular double cover of $\PP^1$ branched at two points $P_0,P_1\in\PP^1$. 
Note that $X\cong\PP^1$ and $\phi_\ast\mcO_X\cong\mcO_{\PP^1}\oplus\mcO_{\PP^1}(-1)$. 
Then $\phi_\ast\mcO_X(2k+1)\cong\mcO_{\PP^1}(k)^{\oplus 2}$ and $\phi_\ast\mcO_X(2k)\cong\mcO_{\PP^1}(k)\oplus\mcO_{\PP^1}(k-1)$ for $k\in\ZZ$. 
\end{lem}
\begin{proof}
We may assume that $P_i\in\PP^1$ is defined by $x_i=0$ for each $i=0,1$, where $x_0,x_1$ are homogeneous coordinates of $\PP^1$. 
Let $U_i$ be the affine open set $\PP^1\setminus\{P_i\}$, and put $x_{ij}:=x_j/x_i$ for $i,j=0,1$. 
Put $\mcM:=\mcO_{\PP^1}^{\oplus2}$, and let $G_{ij}$ be the identity matrix for $i,j=0,1$ as the gluing map $\mcO_{U_j}^{\oplus2}\to\mcO_{U_i}^{\oplus2}$ of $\mcM$. Let $M:\mcM(-1)\to\mcM$ be the morphism represented by 
\begin{align*}
	M=
	\begin{pmatrix}
		0 & x_1 \\ x_0 & 0
	\end{pmatrix}
	:\mcO_{\PP^1}(-1)\oplus\mcO_{\PP^1}(-1)\to\mcO_{\PP^1}\oplus\mcO_{\PP^1}. 
\end{align*}
Then $(\mcM,M)$ is an admissible pair for $\phi$. 
Put $\mcL:=\mcL_{(\mcM,M)}$. 
Since $\dim \HH^0(X,\mcL)=2$ and $X\cong\PP^1$, we have $\mcL\cong\mcO_X(1)$. 
Let $M_i:=M|_{U_i}$ for $i=0,1$, and put $M_i^\natu:=A_i^{-1}M_iA_i$ and $G_{ij}^\natu:=A_i^{-1}G_{ij}A_j$, where 
\begin{align*}
	A_0&:=
	\begin{pmatrix}
		1 & 0 \\ 0 & 1
	\end{pmatrix},
	&
	A_1&:=
	\begin{pmatrix}
		0 & 1 \\ 1 & 0
	\end{pmatrix}
\end{align*}
Then $(\{G_{ij}^\natu\},\{M_i^\natu\})$ is a good representation of $(\mcM,M)$. 
Let $K_{ij}^+$ and ${K}_{ij}^-$ be the matrices $K_{ij}^{(k)+}$ and ${K}_{ij}^{(k)-}$ in (\ref{eq:Kij^k}) and (\ref{eq:-Kij^k}) for $(\{G_{ij}^\natu\},\{M_i^\natu\})$, respectively:
\begin{align*}
	K_{01}^+&=
	\begin{pmatrix}
		0 & x_{01} \\ 1 & 0
	\end{pmatrix},
	&
	{K}_{01}^-&=-\frac{1}{x_{01}}
	\begin{pmatrix}
		0 & x_{01} \\ 1 & 0
	\end{pmatrix}.
\end{align*}
By Theorem~\ref{thm:group_law}, the transition function $G_{01}^{[n]}$ of $\phi_\ast\mcL^{n}$ is given by
\begin{align*}
	G_{01}^{[n]}=\left\{
	\begin{array}{ll}
	(K_{01}^+)^nG_{01}^{(0)}=x_{01}^k
	\begin{pmatrix}
		0 & 1 \\ 1 & 0
	\end{pmatrix} & \mbox{(if $n=2k+1$ for $k\geq0$)}
	\\[1.em]
	(K_{01}^+)^nG_{01}^{(0)}=x_{01}^{k-1}
	\begin{pmatrix}
		x_{01} & 0 \\ 0 & 1
	\end{pmatrix} & \mbox{(if $n=2k$ for $k\geq0$)}
	\\[1.em]
	({K}_{01}^-)^{-n}G_{01}^{(0)}=(-1)^n x_{01}^{k}
	\begin{pmatrix}
		0 & 1 \\ 1 & 0
	\end{pmatrix} & \mbox{(if $n=2k+1$ for $k<0$)}
	\\[1.em]
	({K}_{01}^-)^{-n}G_{01}^{(0)}=(-1)^n x_{01}^{k-1}
	\begin{pmatrix}
		x_{01} & 0 \\ 0 & 1
	\end{pmatrix} & \mbox{(if $n=2k$ for $k<0$)}
	\end{array}
	\right.
\end{align*}
This implies that $\phi_\ast\mcL^{2k+1}\cong\mcO_{\PP^1}(k)^{\oplus2}$ and $\phi_\ast\mcL^{2k}\cong\mcO_{\PP^1}(k)\oplus\mcO_{\PP^1}(k-1)$. 
\end{proof}

From now, let $\phi:X\to\PP^2$ be a double cover branched along smooth conic $B_\phi$, and let $\iota:X\to X$ be the covering transformation of $\phi$. 
Then $X\cong\PP^1\times\PP^1$ and $\pic(X)\cong\ZZ^{\oplus 2}$. 
Hence a line bundle $\mcL$ on $X$ can be represented by bidegree as $\mcL=\mcO_X(k_1,k_2)$. 
If $k_1\leq k_2$, then $\phi_\ast\mcO_X(k_1,k_2)\cong\phi_\ast\mcO_X(0,k_2-k_1)\otimes\mcO_{\PP^2}(k_1)$. 
Thus it is enough to compute the direct summands of $\phi_\ast\mcO_{X}(0,k)|_L$ for $k\geq0$ and lines $L\subset\PP^2$.

\begin{prop}[{\cite[Proposition 8]{schwarzenberger1961}}]\label{prop:schw}
Let $\phi:X\to\PP^2$ be as above. 
\begin{enumerate}[label={\rm (\roman{enumi})}]
	\item If a line $L\subset\PP^2$ intersects transversally with $B_\phi$, then $\phi_\ast\mcO_X(0,2k+1)|_L\cong\mcO_L(k)^{\oplus2}$ and $\phi_\ast\mcO_X(0,2k)|_L\cong\mcO_{L}(k)\oplus\mcO_L(k-1)$ for $k\geq0$. 
	\item If $L\subset\PP^2$ is a tangent line of $B_\phi$, then $\phi_\ast\mcO_X(0,n)|_L\cong\mcO_L(n-1)\oplus\mcO_L$ for $n\geq0$. 
\end{enumerate}
In particular, tangent lines of $B_\phi$ are jumping lines of $\phi_\ast\mcO_X(k_1,k_2)$ if the integers $k_1$ and $k_2$ satisfy $|k_1-k_2|> 2$. 
\end{prop}
\begin{proof}
After a certain projective transformation, we may assume that $B_\phi$ is defined by $F:=x_0^2+x_1x_2=0$, where $[x_0{:}x_1{:}x_2]$ is a system of homogeneous coordinates. 
Then we have a line bundle $\mcL$ on $X$ given by the admissible pair $(\mcM,M)$, where $\mcM=\mcO_{\PP^2}^{\oplus2}$ and $M$ is the morphism represented by the matrix
\begin{align*}
	M:=\begin{pmatrix}
		x_0 & x_2 \\ x_1 & -x_0
	\end{pmatrix}
	:\mcM(-1)\to\mcM.
\end{align*}
Since $\dim\HH^0(X,\mcL)=\dim\HH^0(\PP^2,\mcM)=2$, we obtain either $\mcL\cong\mcO_X(0,1)$ or $\mcL\cong\mcO_X(1,0)$. 
Without loss of generality, we may assume that $\mcL\cong\mcO_X(0,1)$. 
If $L\subset\PP^2$ intersects transversally with $B_\phi$, then $L':=\phi^{-1}(L)\cong \PP^1$ and $\mcL|_{L'}\cong\mcO_{L'}(1)$. 
By Lemmas~\ref{lem:rest_push} and \ref{lem:p1}, we obtain (i). 

To prove (ii), we take a good representation of $(\mcM,M)$. 
Let $U_i$ be the affine open subset $\{x_i\ne0\}$ of $\PP^2$ for each $i=0,1,2$. 
Put $x_{ij}:=x_j/x_i$, and 
\begin{align*}
	G_{ij}&:=\begin{pmatrix} 1 & 0 \\ 0 & 1 \end{pmatrix}, \qquad
	M_i:=M|_{U_i}=\begin{pmatrix} x_{i0} & x_{i2} \\ x_{i1} & -x_{i0} \end{pmatrix} \quad (i,j=0,1,2), 
	\\[0.5em]
	A_0&:=\begin{pmatrix} 1 & 0 \\ 1 & 1 \end{pmatrix}, \qquad 
	A_1:=\begin{pmatrix} 1 & 0 \\ 0 & 1 \end{pmatrix}, \qquad 	
	A_2:=\begin{pmatrix} 0 & 1 \\ 1 & 0 \end{pmatrix}.
\end{align*}
Note that $\{G_{ij}\}$ is the set of transition functions of $\mcM$. 
Let $U_0^\natu$ be the affine open subset $U_0\cap\{-2x_0+x_1-x_2\ne0\}$, and put $U_i^\natu:=U_i$ for $i=1,2$. 
Put $G_{ij}^\natu:=A_i^{-1}G_{ij}A_j$, and $M_i^\natu:=A_i^{-1}M_iA_i$. 
Then $(\{G_{ij}^\natu\},\{M_i^\natu\})_{\{U_i^\natu\}}$ is a good representation of $(\mcM,M)$. 
Let $K_{ij}$ be the matrix $K_{ij}^{(1)+}$ in (\ref{eq:Kij^k}) for $(\{G_{ij}^{(1)}\},\{M_i^{(1)}\})_\gtU=(\{G_{ij}^\natu\},\{M_i^\natu\})_{\{U_i^\natu\}}$. 
We have 
\begin{align*}
	G_{12}^{(0)}=
	\begin{pmatrix}
		1 & 0 \\ 0 & x_{12}^{-1}
	\end{pmatrix},
	\qquad
	K_{12}=
	\begin{pmatrix}
		-x_{10} & x_{10}^2+x_{12} \\ 1 & -x_{10}
	\end{pmatrix}.
\end{align*}
Let $L$ be the tangent line at $[1{:}2b{:}2c]$, where $b,c$ satisfy $4bc=-1$. 
In this case, $L$ is defined by $x_0+cx_1+bx_2=0$. 
Then $L\subset U_1^\natu\cap U_2^\natu$, $L\cap U_1^\natu=\mathrm{Spec}\,\CC[x_{12}]$ and $L\cap U_2^\natu=\mathrm{Spec}\,\CC[x_{21}]$. 
By direct computation, we obtain the matrix $(K_{12}^nG_{12}^{(0)})|_L$ for $n\geq 0$ as follows:
\begin{align*}
	\frac{2^{n-1}}{x_{12}(bx_{12}-c)}
	\begin{pmatrix}
		x_{12}(bx_{12}-c)(b^nx_{12}^n+c^n) & (bx_{12}-c)^2(b^nx_{12}^n-c^n) \\ x_{12}(b^nx_{12}^n-c^n) & (bx_{12}-c)(b^nx_{12}^n+c^n)
	\end{pmatrix}. 
\end{align*}
Thus we have
\begin{align*}
	&\quad Q_1^{-1}\left(K_{12}^nG_{12}^{(0)}\right)|_L\,Q_2=2^n\begin{pmatrix} c^n & 0 \\ 0 & b^nx_{12}^{n-1} \end{pmatrix}, 
	\mbox{ where } 
	\\[0.5em]
	&Q_1^{-1}:=
%	\frac{1}{2c^n(bx_{12}-c)}
%	\begin{pmatrix}
%		2c^n(bx_{12}-c) & -2c^n(bx_{12}-c)^2 \\ -(bx_{12}-c) & (b_{12}-c)(b^nx_{12}^n+c^n)
%	\end{pmatrix}
	\begin{pmatrix}
		1 & -(bx_{12}-c) \\[0.5em] -\dfrac{b^nx_{12}^n-c^n}{2c^n(bx_{12}-c)} & \dfrac{b^nx_{12}^n+c^n}{2c^n}
	\end{pmatrix},
	\quad
	Q_2:=\begin{pmatrix} 1 & b-cx_{21} \\ 0 & 1 \end{pmatrix}. 
\end{align*}
Since $Q_1\in\GL(2,\CC[x_{12}])$ and $Q_2\in\GL(2,\CC[x_{21}])$, this implies that $\phi_\ast\mcL^{\otimes n}|_L\cong\mcO_L(n-1)\oplus\mcO_L$. 

For remaining lines $L_1:=\{x_1=0\}$ and $L_2:=\{x_2=0\}$, take the projective transformation $p:\PP^2\to\PP^2$ given by $[y_0{:}y_1{:}y_2]\mapsto [x_0{:}x_1{:}x_2]=[(y_1+y_2)/2{:}y_0-(y_1-y_2)/2{:}y_0+(y_1-y_2)/2]$. 
Then $p^\ast(x_0^2+x_1x_2)=y_0^2+y_1y_2$, $p^\ast L_1=\{2y_0-y_1+y_2=0\}$ and $p^\ast L_2=\{2y_0+y_1-y_2=0\}$. 
By the above argument, we obtain $\phi_\ast\mcL^{\otimes n}|_{L_i}\cong\mcO_{L_i}(n-1)\oplus\mcO_{L_i}$. 
\end{proof}

\begin{rem}
Let $\phi:X\to\PP^2$ be a non-singular double cover branched along smooth curve of degree $2r$. 
In \cite{schwarzenberger1961}, jumping lines of $\phi_\ast\mcL$ was computed for several line bundles $\mcL$ on $X$ in the case of $r=2$. 
Ottaviani \cite{ottaviani1984} and Vall\`es \cite{valles2009} studied jumping lines of the direct images of line bundles on $X$. 
\end{rem}

For a line bundle $\mcL$ on a non-singular double cover $X$ over $Y$, we obtain global sections of $\mcL$ by computing those of $\phi_\ast\mcL$ since $\Gamma(X,\mcL)\cong\Gamma(Y,\phi_\ast\mcL)$ under the isomorphism of Corollary~\ref{cor:corr}~\ref{cor:corr_morphism}. 

\begin{ex}\label{ex:global_section}
Let $\phi:X\to\PP^2$ be the non-singular double cover in Proposition~\ref{prop:schw}. 
Put $\mcL:=\mcO_X(4,2)$. 
Since $\phi^\ast\mcO_{\PP^2}(1)\cong\mcO_X(1,1)$, we have $\mcL\cong\mcO_X(2,0)\otimes\phi^\ast\mcO_{\PP^2}(2)$. 
Since $\phi_\ast\mcL$ is reflexive, the restriction $\Gamma(\PP^2,\phi_\ast\mcL)\to\Gamma(U_1\cup U_2,\phi_\ast\mcL)$ is isomorphism (cf. \cite{hartshorne1980}). 
Hence it is enough to compute the sections over $U_1\cup U_2$. 
By our proof of Proposition~\ref{prop:schw}, a transition function $G_{12}^\sim$ of $\phi_\ast\mcL$ between $U_1$ and $U_2$ is 
\begin{align*}
	G_{12}^\sim=x_{12}^2K_{12}^2G_{12}^{(0)}=
	\begin{pmatrix}
		-x_{10} & x_{10}^2+x_{12} \\ 1 & -x_{10}
	\end{pmatrix}^2
	\begin{pmatrix}
		x_{12}^2 & 0 \\ 0 & x_{12}
	\end{pmatrix}
\end{align*}
If a section ${}^t(s_1, s_2)\in\Gamma(U_2,\mcO_{\PP^2}^{\oplus2})$ satisfies $G_{12}^\sim{}^t(s_1,s_2)\in\Gamma(U_1,\mcO_{\PP^2}^{\oplus 2})$, then the section ${}^t(s_1,s_2)$ is of the form
\begin{align*}
	s_1&= 2c_1x_{20}^4+2c_2x_{20}^3x_{21}+c_3x_{20}^3+2c_4x_{20}^2x_{21}^2+(c_1+c_5)x_{20}^2x_{21} 
\\ &\quad
+(c_6+c_7)x_{20}^2 +(c_2+c_8)x_{20}x_{21}^2+(c_{9}+c_{10})x_{20}x_{21}
\\ &\quad
+(c_{11}+c_{12})x_{20}+c_{4}x_{21}^3+c_{13}x_{21}^2+c_{14}x_{21}+c_{15}, 
\\[0.5em]
s_2&=2c_1x_{20}^3+2c_2x_{20}^2x_{21}+c_3x_{20}^2+2c_4x_{20}x_{21}^2+c_5x_{20}x_{21}
\\ &\quad
+c_6x_{20}+c_8x_{21}^2+c_9x_{21}+c_{11}
\end{align*}
for $c_1,\dots,c_{15} \in\CC$. 
Thus $s_1,s_2$ as above give a global section of $\phi_\ast\mcL$, and the global section $v$ of $\mcL$ corresponding to it. 
Note that $\deg(s_1)=4$ and $\deg(s_2)=3$. 
On the other hand, the degree of $s_1^2-s_2^2F$ is $6$, which defines the image of the curve on $X$ defined by $v=0$. 
\end{ex}

%\appendix 

%\bibliographystyle{amsplain}
%\bibliography{biblio}

%\noindent
%Taketo SHIRANE\\
%Department of Mathematical Sciences, \\
%Faculty of Science and Technology, \\
%Tokushima University, \\
%2-1 Minamijyousanjima-cho, Tokushima 770-8506, JAPAN. \\
%E-mail: {\tt shirane@tokushima-u.ac.jp}, \\
%ORCID iD:  0000-0002-4531-472X.

\end{document}